\pdfoutput=1
\documentclass[11pt,a4paper]{article}

\usepackage[T1]{fontenc}
\usepackage[utf8]{inputenc}
\usepackage[english]{babel}
\usepackage{amsmath,amssymb,amsfonts,amsthm}
\usepackage{graphicx}
\usepackage{xcolor}
\usepackage[margin=2.6cm]{geometry}
\usepackage{tikz}
\usetikzlibrary{shapes,patterns}
\usepackage{caption}
\usepackage{subcaption}
\usepackage{float}
\usepackage[hidelinks]{hyperref}

\theoremstyle{plain}
\newtheorem{theorem}{Theorem}[section]
\newtheorem{proposition}{Proposition}[section]
\newtheorem{corollary}{Corollary}[section]

\theoremstyle{definition}
\newtheorem{definition}{Definition}[section]
\newtheorem{example}{Example}[section]

\theoremstyle{remark}
\newtheorem{remark}{Remark}[section]

\DeclareMathOperator*{\argmax}{arg\,max}
\DeclareMathOperator*{\gra}{gph}
\DeclareMathOperator*{\nep}{NEP}
\DeclareMathOperator*{\gnep}{GNEP}
\def\R{\mathbb{R}}
\def\nn{\mathbb{N}}
\def\qq{\mathbb{Q}}
\newcommand{\tos}{\rightrightarrows} % point-to-set mappings

\begin{document}

\title{Sequential Stability of the Value Function and the Solution Mapping
in Berge's Maximum Theorem via Variational Convergence
\thanks{This paper has received funding from ANID--Chile through projects
Fondecyt 1220687 (L\'opez) and Fondecyt 1200525 (Fierro).}}

\author{%
John Cotrina\thanks{Universidad del Pac\'ifico, Lima, Per\'u.
\texttt{cotrina\_je@up.edu.pe}}
\and
Ra\'ul Fierro\thanks{Instituto de Matem\'aticas, Pontificia Universidad
Cat\'olica de Valpara\'{\i}so, Valpara\'iso, Chile.
\texttt{raul.fierro@pucv.cl}}
\and
Rub\'en L\'opez\thanks{Departamento de Matem\'atica, Universidad de
Tarapac\'a, Arica, Chile. \texttt{rlopezm@academicos.uta.cl}}
}

\date{January 27, 2026}

\maketitle

\begin{abstract}
Berge's maximum theorem ensures the continuity of the value function and the
upper semicontinuity of the solution mapping in parametric optimization
problems. This theorem plays a central role in optimization theory, game
theory, and dynamic programming. Motivated by the inherent inaccuracies in
optimization data, this paper investigates the stability of such problems under
sequential perturbations of both the objective function and the feasible
mapping. The analysis focuses on the convergence of sequences of value functions
and solution mappings via variational approximations of the data. To this end,
we employ lower and upper continuous, epi- and hypo-convergence notions for
functions, together with lower and upper continuous and graphical convergence
notions for multifunctions. In addition, we study some relationships among these
types of convergence and provide examples and counterexamples associated with
the corresponding notions. Our results extend and complement existing stability
results in the literature. We provide applications to generalized Nash
equilibrium problems, where stability is obtained via a direct approach, as well
as to finite-horizon dynamic programming models under novel perturbation
assumptions.
\end{abstract}



\section{Introduction}\label{intro}

Berge's maximum theorem~\cite{Be63} guarantees the continuity of the value function and
the upper semicontinuity of the solution mapping of a parametric optimization problem. The former property has been used in existence theorems and characterization results for dynamic programming. The latter property has been used when applying Kakutani-type fixed
point theorems.

Berge's maximum theorem holds significant relevance in fields such as economic theory, optimal control, and optimization theory. For instance, in demand
theory, it is concerned with an agent's optimal consumption concerning prices and income, while in capital theory, with the optimal investment strategy based on the existing
capital stock.

The stability of optimization problems is an important property due to the inherent inaccuracy of such problems. The inaccuracy stems from various sources, including predictive
errors in forecasting certain data inputs (such as future demands or returns), measurement errors in empirical data (such as parameters of devices or processes), implementation errors
during computation (including approximations and rounding), and system errors in mathematical modeling. As a result, the solution obtained is, at best, an approximation to the true
solution. In order for this approximate solution to be practically useful, decision-makers require access to stability information about the problem that depends on variations of the
data. See~\cite[Chapter~7]{CPS92} for a detailed explanation of this topic but for linear complementarity problems.

The stability analysis for the parametric optimization problem according to Berge's maximum theorem, consists of studying convergence properties of the sequences of value functions and of solution
mappings when the data of the problem (objective function and feasible mapping) are approximated via convergence notions. Zolezzi~\cite{Zo84}
studied the stability of the value function at a fixed parameter. Lignola and Morgan~\cite{LM92,LM97} studied the stability of the value function and of the solution mapping with respect to
perturbations of the objective function and of the feasible mapping.
Our results extend, complement and shed new
light on the results in~\cite{LM92,LM97,Zo84}.

To approximate the data of the parametric problem, we use lower and upper continuous, epi- and hypo-convergence notions for functions, together with continuous and graphical convergence notions for multifunctions. Some references on these convergence notions are the books of Attouch~\cite{At84}, Beer~\cite{Be93}, Hu and Papageorgiou~\cite{HP97}, Rockafellar and Wets~\cite{RW09}, and Royset and Wets~\cite{RW21}.

As a result of our stability analysis, we establish the continuous and the hypo-convergence of approximate value functions and the outer continuous convergence of approximate solution mappings.

As Berge's maximum theorem holds significant relevance in generalized Nash games, we leverage our findings to examine the approximation of generalized Nash equilibria.
We obtain stability results for generalized Nash games via a direct approach in contrast to \cite{DK24,RW19} where the stability results are obtained for equivalent formulations of generalized Nash games.
Contrary to optimization theory, there are limited studies
addressing the approximation of generalized Nash equilibria. To the best of our knowledge, the works~\cite{CaEtAl86,DK16,GP09,MR99} primarily concentrate on the classical Nash
game, while \cite{DK24,MS08,RW19} address generalized Nash games but in a manner that differs from ours. As another application of our results, we perturb a finite-horizon dynamic programming model by a sequence of plans
and multifunctions representing the feasible sets of plans. We
prove the stability of the model under assumptions that differ from those in~\cite{GT09,G-etal21}.

The paper is organized as follows. In Section~\ref{sec:2}, we fix the notation and recall some preliminaries on variational convergence of functions and multifunctions. Section~\ref{sec:3} is devoted to present our main results. In Section~\ref{sec:4}, we apply our results to generalized Nash equilibrium problems and to finite-horizon dynamic programming models. Finally, in Section~\ref{sec:5}, we present some conclusions.

\section{Notation and preliminaries}\label{sec:2}

We denote by $x=(x_1,\ldots,x_n)$ a vector from $\mathbb{R}^n$, by $y=(y_1,\ldots ,y_m)$ a vector from $\R^m$, by $\mathbb{B}$ the closed unit ball in $\mathbb{R}^n$, and by
$\overline{\R}=\R\cup\{\pm\infty\}$ the extended set of real numbers. For an extended-valued function $f\colon\R^\ell\to\overline{\R}$, we denote
by $\mathrm{dom}_Uf=\{x\in\R^\ell\colon f(x)>-\infty\}$ its upper domain; by $\mathrm{dom}_Lf=\{x\in\R^\ell\colon f(x)<+\infty\}$ its lower domain; by
$\mathrm{epi}\,f:=\{(x,\lambda)\in \R^{\ell+1}\colon f(x)\le\lambda\}$ its epigraph; and by $\mathrm{hyp}\,f:=\{(x,\lambda)\in \R^{\ell+1}\colon \lambda\le f(x)\}$ its hypograph. We
say
that $f$ is upper proper (resp. lower proper) if $f(x)<+\infty$ (resp. $f(x)>-\infty$) for all $x\in\R^\ell$ and $\mathrm{dom}_U f\ne\emptyset$  (resp. $\mathrm{dom}_L f\ne\emptyset$). Clearly,
$\mathrm{hyp}\,f$ (resp. $\mathrm{epi}\,f$) is nonempty iff $\mathrm{dom}_U f$ (resp. $\mathrm{dom}_L f$) is nonempty. Therefore, when dealing with the hypograph (resp. epigraph) of a function, we assume that the function has a nonempty upper (resp. lower) domain.

We recall convergence notions for sets from~\cite{RW09}. For a sequence of sets $\{C_k\}$ from $\mathbb{R}^\ell$,
$\limsup\nolimits_kC_k:=\{x\in\R^\ell\colon\exists\, x^{k_j}\in C_{k_j}\;\mbox{s.t.}\; x^{k_j}\to x\}$ is its outer limit and $\liminf\nolimits_kC_k:=\{x\in\R^\ell\colon\exists\,
x^{k}\in
C_{k}\;\mbox{s.t.}\; x^{k}\to x\}$ is its inner limit, where $\{x^{k_j}\}$ is a subsequence of $\{x^k\}$. We say that $\{C_k\}$ converges in the sense of
Painlev\'e--Kuratowski to $C$, denoted by $C_k\to C$ or $\lim_k C_k=C$,  if $\limsup_kC_k=C=\liminf_kC_k$, or equivalently if
$\limsup_kC_k\subset C\subset\liminf_kC_k$.

A sequence $\{C_k\}$ is said to be: eventually bounded if there exists $N\in\mathbb{N}$ such that $\bigcup_{k\ge N}C_k$ is bounded; strongly eventually bounded if it
is eventually bounded and $\bigcup_{k\ge j}C_k\ne\emptyset$ for all $j\in\nn$; nonempty-valued if $C_k\ne\emptyset$ for all $k\in\nn$; and eventually nonempty-valued (resp.
empty-valued) if there exists~$N\in\mathbb{N}$ such that $C_k\ne\emptyset$ (resp. $C_k=\emptyset$) for all $k\ge N$. Clearly, nonempty-valuedness implies
eventually nonempty-valuedness which in turn implies that $\bigcup\nolimits_{k\ge j}C_k\ne\emptyset$ for all $j\in\nn$. If $\{C_k\}$ is eventually bounded and eventually
nonempty-valued, then it is strongly eventually bounded. The following conditions are equivalent:
\begin{itemize}
\item $\bigcup\nolimits_{k\ge j}C_k\ne\emptyset$, for all $j\in\nn$;
\item $\{C_k\}$ is not eventually empty-valued;
\item There exists a subsequence $\{C_{k_j}\}$ such that $C_{k_j}\ne\emptyset$, for all $j\in\nn$.
\end{itemize}
We establish conditions under which $\limsup_k C_k$ is a nonempty compact set. To do this, we recall the `escaping to the horizon' property $C_k\to\emptyset$ (or equivalently
$\limsup\nolimits_k C_k=\emptyset$). The following assertions are equivalent (see~\cite[Corollary~4.11]{RW09}):
\begin{itemize}
\item $\limsup\nolimits_k C_k=\emptyset$;
\item For every $\rho>0$ there exists $N\in\mathbb{N}$ such that $C_{k}\cap \rho\mathbb{B}=\emptyset$, for all $k\ge N$;
\item $d_{C_k}(0)\to +\infty$.
\end{itemize}
Here $d_{C_k}(0)$ is the distance from $0$ to $C_k$. From this, we infer that $\limsup\nolimits_k C_k\ne\emptyset$ iff there exists $\rho>0$ such that $\bigcup\nolimits_{k\ge j}(C_k\cap \rho\mathbb{B})\ne\emptyset$, for all $j\in\nn$.

\begin{proposition}\label{prop:p2}
If $\{C_k\}$ is strongly eventually bounded, then
\begin{equation}\label{eq:limsup}
\limsup\nolimits_k C_k\mbox{ is nonempty and compact}.
\end{equation}
Moreover, for every $\varepsilon>0$ there exists $N\in\mathbb{N}$ such that
\begin{equation}\label{eq:inclusion}
\bigcup\nolimits_{k\ge N}C_k\subset \{x\in\R^\ell\colon d_{\, \limsup_k\! C_k}(x)<\varepsilon\}.
\end{equation}
\end{proposition}

\begin{proof}
We prove the first part. Let us denote $A_j:=\bigcup_{k\ge j}C_k$ for $j\in\nn$. By hypothesis, $A_j$ is nonempty for all $j\in\nn$ and bounded for $j$ large enough.
Hence
$\mathrm{cl}\,A_j$ is nonempty and compact for $j$ large enough. As $\mathrm{cl}\,A_{j+1}\subset \mathrm{cl}\,A_j$ for all~$j\in\nn$, by Cantor's intersection theorem, we infer that
$\bigcap_{j}\mathrm{cl}\,A_j$ is nonempty and compact. The result follows since $\bigcap_{j}\mathrm{cl}\,A_j=\limsup_k C_k$ (see~\cite[Exercise~4.2$(b)$]{RW09}).

We prove the second part. Let us denote $C_\varepsilon:= \{x\in\R^\ell\colon d_{\,\limsup_n\! C_n}(x)<\varepsilon\}$. As $\bigcap_{j}\mathrm{cl}\,A_j\subset C_\varepsilon$, we have
$(\bigcap_{j}\mathrm{cl}\,A_j)\cap C^{\mathrm{c}}_\varepsilon=\emptyset$; i.e., $\bigcap_{j}(\mathrm{cl}\,A_j\cap C^{\mathrm{c}}_\varepsilon)=\emptyset$ with the sets in parentheses
being closed. By Cantor's intersection theorem, there exists $N$ such that $\mathrm{cl}\,A_{N}\cap C^{\mathrm{c}}_\varepsilon=\emptyset$; i.e.,
$\bigcup\nolimits_{k\ge N}C_k\subset C_\varepsilon$.
\end{proof}

\begin{remark}
\begin{enumerate}
	\item Condition~\eqref{eq:limsup} in Proposition~\ref{prop:p2} does not imply that
$\{C_k\}$ is eventually bounded.
Indeed, for $C_k=[0,1]\cup\{k\}$ for all $k$, we have $\lim_k C_k=[0,1]$ but $\{C_k\}$ is not eventually bounded.

\item Condition~\eqref{eq:inclusion} does not guarantee that $\{C_k\}$ is eventually bounded either. This is demonstrated by setting $C_k=\R^\ell$ for all $k$.
\end{enumerate}
\end{remark}

We recall notions of set-valued analysis from~\cite{RW09}. A multifunction $\Phi:\R^m\tos\R^n$ is a mapping that associates to any vector
$y\in \R^m$ a set $\Phi(y)\subset\R^n$. We denote by $\mathrm{dom}\,\Phi:=\{y\in \R^m\colon \Phi(y)\ne\emptyset\}$ its domain and by $\mathrm{gph}\,\Phi:=\{(y,x)\in \R^{m+n}\colon
x\in\Phi(y)\}$ its graph. A mapping $\Phi$ is said to be proper if it has a nonempty domain. For the multifunction $\Psi\colon\R^\ell\tos\R^m$, the composition $\Phi\circ\Psi\colon\R^\ell\tos\R^n$ of $\Phi$ and $\Psi$ at~$x$ is defined by $(\Phi\circ\Psi)(x):=\bigcup_{y\in\Psi(x)}\Phi(y)$.
We say that $\Phi$ is: outer semicontinuous (osc) if
$\limsup_k\Phi(y^k)\subset\Phi(y)$ for all $y^k\to y$; inner semicontinuous (isc) if $\Phi(y)\subset\liminf_k\Phi(y^k)$ for all $y^k\to y$; upper semicontinuous (usc) if for any $y$ and any open set $V$ containing $\Phi(y)$ there is a neighborhood~$U$ of $y$ such that $\Phi(U)\subset V$; lower semicontinuous (lsc) if for any $y$ and any open set $V$ with $\Phi(y)\cap V\ne\emptyset$ there is a neighborhood $U$ of $y$ such that $\Phi(z)\cap V\ne\emptyset$ for every $z\in U$; continuous if it is osc
and isc; $K$-continuous if it is usc and lsc; locally bounded if for any $y$ there is a neighborhood $V$ of $y$ such that $\Phi(V)$ is bounded; $N$-valued if $\Phi(y)$ has property $N$ for every $y$ (e.g. nonempty-valued, closed-valued, convex-valued); and
uniformly bounded if $\Phi(\R^m)$ is bounded. For $\Phi,\Psi\colon\R^m\tos\R^n$, we write $\Phi\subset\Psi$ if $\Phi(y)\subset\Psi(y)$ for all $y$ and $\Phi=\Psi$ if
$\Phi\subset\Psi$ and $\Psi\subset\Phi$.

\begin{quote}
{\sc Assumption:} Unless otherwise stated, we assume that $u\colon\R^{n+m}\to\overline{\R}$ is an extended-valued function with nonempty upper (resp. lower) domain when dealing with its hypograph (resp. epigraph) and $\Phi\colon \R^m\tos \R^n$ is a proper
multifunction.
\end{quote}

We study the stability of the parametric optimization problem according to Berge's maximum theorem:
\begin{equation}\tag{$\mathcal{P}_y$}\label{eq:Py}
v(y):=\sup_{x\in\Phi(y)}u(x,y),
\end{equation}
where $y\in \R^m$ is a parameter vector, $u$ is an objective function and $\Phi$ is a feasible
mapping. The function $v\colon\R^{m}\to\overline{\R}$ is called the value function and the solution set of this problem defines the solution mapping $S\colon \R^m\tos \R^n$; i.e.,
\[
S(y):=\argmax_{x\in\Phi(y)}\;u(x,y)=\{x\in\Phi(y)\colon v(y)=u(x,y)\}.
\]
When $\Phi(y)=\emptyset$, as usual, we set $v(y)=-\infty$ and $S(y)=\emptyset$.

We recall a result on the continuity of the value function and of the solution mapping. It is extended by us in Corollary~\ref{cor:berge-variant}.

\begin{proposition}[{\cite[Propositions~3.1--3.4]{HP97}}]\label{prop:berge-prop}
\begin{description}
\item[$(a)$] If $u\colon\R^{n+m}\to\overline{\R}$ is lsc and $\Phi\colon \R^m\tos \R^n$ is nonempty-valued and lsc, then $v\colon\R^{m}\to\overline{\R}$ is lsc.
\item[$(b)$] If $u\colon\R^{n+m}\to\overline{\R}$ is usc and $\Phi\colon \R^m\tos \R^n$ is nonempty-valued, compact-valued and usc, then
  $v\colon\R^{m}\to\overline{\R}$ is usc.
\item[$(c)$] $($Berge's maximum theorem$)$ If $u\colon\R^{n+m}\to\R$ is continuous and $\Phi\colon \R^m\tos \R^n$ is nonempty-valued, compact-valued and $K$-continuous, then $v\colon\R^{m}\to \overline{\R}$ is
  continuous and $S\colon \R^m\tos \R^n$ is nonempty-valued, compact-valued and usc.
\end{description}
\end{proposition}

We study the behavior of the value function $v$ and of the solution mapping $S$
when the data $u$ and $\Phi$ are approximated by $\{u^k\}$ and~$\{\Phi^k\}$ via variational convergence notions. We denote the
respective
approximations of $v$ by~$\{v^k\}$ and of $S$ by~$\{S^k\}$:
\begin{equation}\label{eq:aprox-val-funct-sol-mult}
v^k(y):=\sup_{x\in\Phi^k(y)}u^k(x,y)\;\;\mbox{ and }\;\; S^k(y):=\{x\in\Phi^k(y)\colon v^k(y)=u^k(x,y)\}.
\end{equation}
We will establish convergence results for sequences $\{v^k\}$ and~$\{S^k\}$. To do this, we recall hypo- and epi-limits from~\cite{RW09}.

\begin{definition}[{\cite{RW09}}]\label{def:h-e-limits}  For a sequence $f^k\colon \R^\ell\to\overline{\R}$, its
\begin{itemize}
\item lower hypo-limit is the
function $h\mbox{-}\!\liminf_k f^k$ whose hypograph is
\begin{eqnarray*}
	\mathrm{hyp}(h\mbox{-}\!\liminf\nolimits_k f^k)&=&\liminf\nolimits_k\mathrm{hyp}\,f^k.
\end{eqnarray*}
\item upper hypo-limit is the
function $h\mbox{-}\!\limsup_k f^k$ whose hypograph is
\begin{eqnarray*}
	\mathrm{hyp}(h\mbox{-}\!\limsup\nolimits_k f^k)&=&\limsup\nolimits_k\mathrm{hyp}\,f^k.
\end{eqnarray*}
\item lower epi-limit is the function $e\mbox{-}\!\liminf_k f^k$ whose epigraph is
\begin{eqnarray*}
	\mathrm{epi}(e\mbox{-}\!\liminf\nolimits_k f^k)&=&\limsup\nolimits_k \mathrm{epi}\,f^k.
\end{eqnarray*}
\item upper epi-limit is the function $e\mbox{-}\!\limsup_k f^k$ whose epigraph is
\begin{eqnarray*}
	\mathrm{epi}(e\mbox{-}\!\limsup\nolimits_k
	f^k)&=&\liminf\nolimits_k \mathrm{epi}\,f^k.
\end{eqnarray*}
\end{itemize}
\end{definition}

\begin{remark}
We point out that $\mathrm{hyp}\,f$ $($resp. $\mathrm{epi}\,f$$)$ is nonempty iff $\mathrm{dom}_Uf$ $($resp. $\mathrm{dom}_Lf$$)$ is nonempty. In addition, if $f$ is upper $($resp. lower$)$ proper, then $(x,f(x))\in\mathrm{hyp}\,f$ $($resp.  $(x,f(x))\in\mathrm{epi}\,f$$)$ for all
$x\in\mathrm{dom}_Uf$ $($resp. $x\in\mathrm{dom}_Lf$$)$.
\end{remark}

We recall some properties of hypo- and epi-limits (see~\cite{RW09}). Clearly,
$$h\mbox{-}\!\liminf\nolimits_k f^k\le h\mbox{-}\!\limsup\nolimits_k f^k\;\;\mbox{ and }\;\; e\mbox{-}\!\liminf\nolimits_k f^k\le e\mbox{-}\!\limsup\nolimits_k f^k.$$
Hypo- and epi-limits are related as follows:
$$e\mbox{-}\!\limsup\nolimits_k f^k=-h\mbox{-}\!\liminf\nolimits_k (-f^k)\;\;\mbox{ and }\;\;
e\mbox{-}\!\liminf\nolimits_k f^k=-h\mbox{-}\!\limsup\nolimits_k (-f^k).$$
From this, we infer that $f^k\stackrel{e}{\to}f$ iff $-f^k\stackrel{h}{\to}-f$.

It is clear that
\begin{eqnarray*}
	f\le h\mbox{-}\!\liminf\nolimits_k f^k& \;\;\Longleftrightarrow\;\; &\mathrm{hyp}\,f \subset\liminf\nolimits_k \mathrm{hyp}\,f^k,\\
	h\mbox{-}\!\limsup\nolimits_k f^k\le f& \;\;\Longleftrightarrow\;\; &\limsup\nolimits_k \mathrm{hyp}\,f^k\subset\mathrm{hyp}\,f,\\
   	f\le e\mbox{-}\!\liminf\nolimits_k f^k& \;\;\Longleftrightarrow\;\; &\limsup\nolimits_k \mathrm{epi}\,f^k\subset\mathrm{epi}\,f,\\
    e\mbox{-}\!\limsup\nolimits_k f^k\le f& \;\;\Longleftrightarrow\;\; &\mathrm{epi}\,f\subset\liminf\nolimits_k \mathrm{epi}\,f^k.
\end{eqnarray*}
The following pointwise formulas hold:
$$
\begin{array}{lll}
	(h\mbox{-}\!\liminf\nolimits_k f^k)(x) &=& \max\{\alpha\in\overline{\R}\colon \exists x^k\to x\mbox{ with }\liminf\nolimits_k f^k(x^k)=\alpha\}, \\
	(h\mbox{-}\!\limsup\nolimits_k f^k)(x) &=& \max\{\alpha\in\overline{\R}\colon \exists x^k\to x\mbox{ with }\limsup\nolimits_kf^k(x^k)=\alpha\},\\
	(e\mbox{-}\!\liminf\nolimits_k f^k)(x) &=& \min\{\alpha\in\overline{\R}\colon \exists x^k\to x\mbox{ with }\liminf\nolimits_k f^k(x^k)=\alpha\}, \\
	(e\mbox{-}\!\limsup\nolimits_k f^k)(x) &=& \min\{\alpha\in\overline{\R}\colon \exists x^k\to x\mbox{ with }\limsup\nolimits_kf^k(x^k)=\alpha\}.
\end{array}
$$
The following equivalences hold:
\begin{eqnarray}\label{eq:h-liminf}
	f\le h\mbox{-}\!\liminf\nolimits_k f^k &\;\;\Longleftrightarrow\;\;& \forall x\in\R^{\ell},\exists x^k\to x\colon f(x)\le \liminf\nolimits_k f^k(x^k), \\ \label{eq:h-limsup}
	h\mbox{-}\!\limsup\nolimits_k f^k\le f &\;\;\Longleftrightarrow\;\;&\forall x\in\R^{\ell},\forall x^k\to x\colon \limsup\nolimits_k f^k(x^k) \le f(x),\\ \label{eq:e-liminf}
	f\le e\mbox{-}\!\liminf\nolimits_k f^k &\;\;\Longleftrightarrow\;\;& \forall x\in\R^{\ell},\forall x^k\to x\colon f(x)\le \liminf\nolimits_k f^k(x^k), \\ \label{eq:e-limsup}
	e\mbox{-}\!\limsup\nolimits_k f^k\le f &\;\;\Longleftrightarrow\;\;&\forall x\in\R^{\ell},\exists x^k\to x\colon \limsup\nolimits_k f^k(x^k) \le f(x).
\end{eqnarray}
From this, we infer that
\begin{eqnarray*}
	f\le e\mbox{-}\!\liminf\nolimits_k f^k &\;\;\Longrightarrow\;\;& f\le h\mbox{-}\!\liminf\nolimits_k f^k,\\
	h\mbox{-}\!\limsup\nolimits_k f^k\le f  &\;\;\Longrightarrow\;\;&
	e\mbox{-}\!\limsup\nolimits_k f^k \le f.
\end{eqnarray*}
We define three variational convergence notions by using hypo- and epi-limits.

\begin{definition}[{\cite{RW09}}]\label{def:h-e-c-convergence}  A sequence $f^k\colon \R^\ell\to\overline{\R}$ is said to
\begin{itemize}
	\item hypo-converge to~$f$, denoted by $f^{k}\overset{h}{\rightarrow}f$, if
	$h\mbox{-}\!\liminf\nolimits_k f^k=f= h\mbox{-}\!\limsup\nolimits_k f^k$.
	\item epi-converge to $f$, denoted by $f^{k}\overset{e}{\rightarrow}f$, if
	$e\mbox{-}\!\liminf\nolimits_k f^k=f=e\mbox{-}\!\limsup\nolimits_k f^k$.
	\item converge continuously to $f$, denoted by
	$f^{k}\overset{c}{\rightarrow}f$, if $f^k(x^k)\to f(x)$ for every $x^k\to x$.
\end{itemize}
\end{definition}

Clearly
$$f^k\stackrel{h}{\to}f \;\Longleftrightarrow \;\mathrm{hyp}\,f^k\to \mathrm{hyp}\,f \;\Longleftrightarrow\;
\begin{cases}
	f\le h\mbox{-}\!\liminf\nolimits_k f^k,\\
	h\mbox{-}\!\limsup\nolimits_k f^k\le f,
\end{cases}
$$
$$f^k\stackrel{e}{\to}f \;\Longleftrightarrow \;\mathrm{epi}\,f^k\to \mathrm{epi}\,f \;\Longleftrightarrow\;
\begin{cases}
	f\le e\mbox{-}\!\liminf\nolimits_k f^k,\\
	e\mbox{-}\!\limsup\nolimits_k f^k\le f.
\end{cases}$$

\begin{remark}\label{rem:h-conv}
	\begin{enumerate}
		\item If $f^k\stackrel{h}{\to}f$ $($resp. $f^k\stackrel{e}{\to}f$, $f^k\stackrel{c}{\to}f$$)$, then $f$ is usc $($resp. lsc, continuous$)$ (see~\cite{RW09}).
		\item If either $f^k\stackrel{h}{\to}f$ or $f^k\stackrel{e}{\to}f$, then for every $x\in\R^\ell$ there exists $x^k\to x$ such that $f^k(x^k)\to f(x)$.
		\item By~\eqref{eq:h-limsup}--\eqref{eq:e-liminf} and above, we have
        $$f^{k}\overset{c}{\rightarrow}f\;\Longleftrightarrow\; h\mbox{-}\!\limsup\nolimits_k f^k\le f\le e\mbox{-}\!\liminf\nolimits_k f^k \;\Longleftrightarrow\; f^{k}\overset{h}{\rightarrow}f\mbox{ and }f^{k}\overset{e}{\rightarrow}f.$$
	\end{enumerate}
\end{remark}

\begin{proposition}\label{prop:c-convergence-subseq}
$f^{k}\overset{c}{\rightarrow}f$ iff $f^{k_j}(x^{k_j})\to f(x)$ for any subsequence $x^{k_j}\to x$.
\end{proposition}

\begin{proof}
Implication $(\Leftarrow)$ is trivial. We check implication $(\Rightarrow)$. For a subsequence $x^{k_j}\to x$, we define a sequence $\{\tilde
x^k\}$ by $\tilde x^1=x^{k_1}$, $\tilde x^2=x^{k_1}$, $\ldots$, $\tilde x^{k_1-1}=x^{k_1}$, $\tilde x^{k_1}=x^{k_1}$, $\tilde x^{k_1+1}=x^{k_2}$, $\ldots$, $\tilde
x^{k_2-1}=x^{k_2}$, $\tilde
x^{k_2}=x^{k_2}$, $\tilde x^{k_2+1}=x^{k_3}$, $\ldots$. Clearly, $\tilde x^k\to x$ and as $f^{k}(\tilde x^{k})\to f(x)$, we have $f^{k_j}(x^{k_j})\to f(x)$.
\end{proof}

Similarly, we can write equivalences
\eqref{eq:h-limsup}--\eqref{eq:e-liminf} in terms of subsequences as follows:
\begin{eqnarray*}
h\mbox{-}\!\limsup\nolimits_k f^k \! \le\! f(x) \; \Longleftrightarrow\; \limsup\nolimits_j f^{k_j}(x^{k_j})\!\le\! f(x),\forall x^{k_j}\to x,\\
f\le e\mbox{-}\!\liminf\nolimits_k f^k \;  \Longleftrightarrow\;
f(x)\le \liminf\nolimits_j f^{k_j}(x^{k_j}),\forall x^{k_j}\to x.
\end{eqnarray*}

We recall the outer and inner graphical limits of a sequence of multifunctions. We then define outer and inner continuous limits, which split continuous convergence into two parts.

\begin{definition}[{\cite{RW09}}]\label{def:g-convergence}
	For a sequence $\Psi^k\colon \R^\ell\tos \R^n$, its
	\begin{itemize}
		\item graphical outer limit is the mapping $g\mbox{-}\!\limsup_k\Psi^k$ whose graph is
		\begin{eqnarray*}
			\mathrm{gph}(g\mbox{-}\!\limsup\nolimits_k\Psi^k)&=&\limsup\nolimits_k \mathrm{gph}\,\Psi^k.
		\end{eqnarray*}
		\item graphical inner limit is the mapping  $g\mbox{-}\!\liminf_k\Psi^k$ whose graph is
		\begin{eqnarray*}
			\mathrm{gph}(g\mbox{-}\!\liminf\nolimits_k\Psi^k)&=&\liminf\nolimits_k\mathrm{gph}\,\Psi^k.
		\end{eqnarray*}
		\item continuous outer limit is the mapping
		$c\mbox{-}\!\limsup_k\Psi^k$ defined by
		\begin{eqnarray*}
			(c\mbox{-}\!\limsup\nolimits_k\Psi^k)(y)&:=&\bigcup\nolimits_{\{y^k\to y\}}\limsup\nolimits_k\Psi^k(y^k).
		\end{eqnarray*}
		\item continuous inner limit is the mapping
		$c\mbox{-}\!\liminf_k\Psi^k$ defined by
		\begin{eqnarray*}
			(c\mbox{-}\!\liminf\nolimits_k\Psi^k)(y)&:=&\bigcap\nolimits_{\{y^k\to y\}}\liminf\nolimits_k\Psi^k(y^k).
		\end{eqnarray*}
	\end{itemize}
\end{definition}

\begin{remark}
We point out that $\mathrm{gph}\,\Psi$ is nonempty iff $\Psi$ is proper.
\end{remark}

We recall some properties of graphical outer and inner limits from~\cite{LS21,RW09}.
$$g\mbox{-}\!\liminf\nolimits_k\Psi^k\subset g\mbox{-}\!\limsup\nolimits_k \Psi^k\;\;\mbox{ and }\;\; c\mbox{-}\!\liminf\nolimits_k\Psi^k\subset c\mbox{-}\!\limsup\nolimits_k \Psi^k.$$
By definition, we have
\begin{eqnarray*}
	\Psi \subset g\mbox{-}\!\liminf\nolimits_k\Psi^k&\;\;\Longleftrightarrow\;\;&\mathrm{gph}\,\Psi \subset\liminf\nolimits_k (\mathrm{gph}\,\Psi^k),\\ g\mbox{-}\!\limsup\nolimits_k\Psi^k\subset\Psi&\;\;\Longleftrightarrow\;\;&\limsup\nolimits_k (\mathrm{gph}\,\Psi^k)\subset\mathrm{gph}\,\Psi.
\end{eqnarray*}
The following pointwise formulas hold:
\begin{eqnarray*}
	(g\mbox{-}\!\liminf\nolimits_k\Psi^k)(y) &=& \bigcup\nolimits_{\{y^k\to y\}}\liminf\nolimits_k\Psi^k(y^k), \\
	(g\mbox{-}\!\limsup\nolimits_k\Psi^k)(y) &=& \bigcup\nolimits_{\{y^k\to y\}}\limsup\nolimits_k\Psi^k(y^k).
\end{eqnarray*}
From these formulas, we obtain
$$c\mbox{-}\!\liminf\nolimits_k\Psi^k\subset g\mbox{-}\!\liminf\nolimits_k\Psi^k\;\; \mbox{ and }\;\; c\mbox{-}\!\limsup\nolimits_k\Psi^k=g\mbox{-}\!\limsup\nolimits_k\Psi^k.$$
By this equality, from now on, we will use only the outer limit notation $g\mbox{-}\!\limsup\nolimits_k\Psi^k$ and not the outer continuous limit notation.

The following equivalences hold:
\begin{eqnarray}\label{eq:g-liminf}
	\Psi \subset g\mbox{-}\!\liminf\nolimits_k\Psi^k &\;\Longleftrightarrow\;& \forall y\in\R^\ell\colon \Psi(y)\subset  \bigcup\nolimits_{\{y^k\to y\}}\liminf\nolimits_k \Psi^k(y^k), \\  \label{eq:g-limsup}
	g\mbox{-}\!\limsup\nolimits_k\Psi^k\subset \Psi &\;\Longleftrightarrow\;& \forall y\in\R^\ell,\forall y^k\to y\colon \limsup\nolimits_k \Psi^k(y^k) \subset \Psi(y),\\ \label{eq:c-liminf-phi}
	\Psi \subset c\mbox{-}\!\liminf\nolimits_k\Psi^k &\;\;\Longleftrightarrow\;\;& \forall y\in\R^\ell,\forall y^k\to y\colon \Psi(y)\subset \liminf\nolimits_k \Psi^k(y^k).
\end{eqnarray}
We recall continuous and graphical convergence notions for multifunctions.

\begin{definition}[{\cite{RW09}}] A sequence $\Psi^k\colon \R^\ell\tos \R^n$ is said to
\begin{itemize}
\item graph converge to~$\Psi$, denoted by $\Psi^{k}\overset{g}{\rightarrow}\Psi$, if
$g\mbox{-}\!\liminf\nolimits_k\Psi^k=\Psi= g\mbox{-}\!\limsup\nolimits_k \Psi^k$.
\item converge continuously to~$\Psi$, denoted by
$\Psi^k\stackrel{c}{\to}\Psi$, if $\Psi^k(y^k)\to \Psi(y)$ for every $y^k\to y$.
\end{itemize}
\end{definition}

\begin{remark}\label{rem:continuous2}
\begin{enumerate}
\item Lignola and Morgan~\cite{LM92,LM97} defined upper and lower convergence of $\{\Psi^k\}$ to $\Psi$ by using the right-sides of~\eqref{eq:g-limsup} and~\eqref{eq:c-liminf-phi}, respectively.
Whereas $G^-$ and $G^+$ convergence in~\cite{LM97} were defined by the right-sides of~\eqref{eq:g-liminf} and~\eqref{eq:g-limsup}, respectively. In this paper, we adhere to the recent terminology on variational convergence for multifunctions in~\cite{RW09}.
\item If $\Psi^k\stackrel{c}{\to}\Psi$ $($resp. $\Psi^k\stackrel{g}{\to}\Psi$$)$, then $\Psi$ is continuous  $($resp. osc$)$ and thus closed-valued $($see~\cite{RW09}$)$.
\item Clearly $\Psi^k\stackrel{c}{\to}\Psi$ iff $c\mbox{-}\!\liminf\nolimits_k\Psi^k=\Psi= g\mbox{-}\!\limsup\nolimits_k \Psi^k$.
\item By the pointwise formulas, we infer that $\Psi^k\stackrel{c}{\to}\Psi$ implies $\Psi^k\stackrel{g}{\to}\Psi$.
\end{enumerate}
\end{remark}

Clearly
$$\Psi^{k}\overset{g}{\rightarrow}\Psi\;\;\Longleftrightarrow\;\;\mathrm{gph}\, \Psi^{k}\rightarrow \mathrm{gph}\,\Psi\;\;\Longleftrightarrow\;\;
\begin{cases}\Psi \subset g\mbox{-}\!\liminf\nolimits_k\Psi^k,\\ g\mbox{-}\!\limsup\nolimits_k\Psi^k\subset\Psi.
\end{cases}$$
Moreover
$$\Psi^k\stackrel{c}{\to}\Psi \;\;\Longleftrightarrow\;\;
\begin{cases}
	\Psi \subset c\mbox{-}\!\liminf\nolimits_k\Psi^k\\ g\mbox{-}\!\limsup\nolimits_k\Psi^k\subset\Psi.
\end{cases}$$
Let us consider the following mappings (see~\cite{LM92}):
\begin{itemize}
	\item Interval mappings: $\Psi^k_1(y):=[f^k(y),g^k(y)]$ and $\Psi_1(y):=[f(y),g(y)]$ with $f^k,g^k,f,g\colon\R^n\to\overline{\R}$ such that $f^k\le g^k$ for all $k$ and $f\le g$. Let $\overset{\circ}{\Psi}\,^{k}_1$ and $\overset{\circ}{\Psi}_1$ be the same mappings but with open intervals instead.
	\item Unexplicit mappings: $\Psi^k_2(y)\equiv C_k$ and $\Psi_2(y)\equiv C$ with $C_k,C\subset\R^m$ for all $k$.
	\item Inequality mappings:
$\Psi^k_3(y):=\{x\in\R^n\colon g_i^k(x,y)\le 0, i=1,\ldots ,p\}$ and
$\Psi_3(y):=\{x\in\R^n\colon g_i(x,y)\le 0, i=1,\ldots ,p\}$,
with $g_i,g_i^k\colon \R^{n+m}\to\overline{\R}$ for all $k$ and $i=1,\ldots ,p$. Let $\overset{\circ}{\Psi}\,^{k}_3$ and $\overset{\circ}{\Psi}_3$ be the same mappings but with strict inequalities instead.
\end{itemize}
We study the convergence of these mappings. Part~$(b)$ and the first implications of~$(a)$ and~$(c)$ appear in~\cite[Example~2.2.2, Remark~2.2.2]{LM92}. In~\cite[Propositions~3.3.1 and 3.3.3]{LM92} there are convergence results for inequality mappings but under other convergence notions.

\begin{proposition}\label{prop:ex-feas-mappings}
\begin{description}
\item[$(a)$] For interval mappings one has
\begin{eqnarray*}
			f\le e\mbox{-}\!\liminf\nolimits_k f^k\;\,\mbox{ and }\,\; h\mbox{-}\!\limsup\nolimits_k g^k\le g & \;\;\Longrightarrow\;\;& g\mbox{-}\!\limsup\nolimits_k\Psi^k_1\subset\Psi_1,\\
			h\mbox{-}\!\limsup\nolimits_k f^k\le f\;\,\mbox{ and }\,\; g \le e\mbox{-}\!\liminf\nolimits_k g^k &\;\;\Longrightarrow\;\; & \overset{\circ}{\Psi}_1 \subset c\mbox{-}\!\liminf\nolimits_k\overset{\circ}{\Psi}\,^{k}_1.
\end{eqnarray*}

\item[$(b)$] For unexplicit mappings one has
\begin{eqnarray*}
	C \!\subset\!\liminf\nolimits_k C_k 	&\;\; \Longleftrightarrow \;\; & \Psi_2 \subset c\mbox{-}\!\liminf\nolimits_k\Psi^k_2,\\
	\limsup\nolimits_k C_k \subset C &\;\; \Longleftrightarrow \;\; & 	g\mbox{-}\!\limsup\nolimits_k\Psi^k_2\subset\Psi_2.
\end{eqnarray*}

\item[$(c)$] For inequality mappings one has
\begin{eqnarray*}
			g_i\le e\mbox{-}\!\liminf\nolimits_k g^k_i,\, i=1,\ldots, p&\;\; \Longrightarrow \;\;&
			g\mbox{-}\!\limsup\nolimits_k\Psi^k_3 \subset\Psi_3,\\
			h\mbox{-}\!\limsup\nolimits_k g^k_i\le g_i,\, i=1,\ldots, p&\;\; \Longrightarrow \;\;&
			\overset{\circ}{\Psi}_3 \subset c\mbox{-}\!\liminf\nolimits_k \overset{\circ}{\Psi}\,^{k}_3.
\end{eqnarray*}
	\end{description}
\end{proposition}

\begin{proof} $(a)$ The first implication appears in~\cite[Example~2.2.2]{LM92}.
We check the second one. If $x\in\overset{\circ}{\Psi}_1(y)$, then $f(y)< x< g(y)$. Let $y^k\to y$ be fixed. By~\eqref{eq:h-limsup}--\eqref{eq:e-liminf} there exists $N$ such that $f^k(y^k)< x< g^k(y^k)$ for all $k\ge N$; i.e.  $x\in\liminf_k\overset{\circ}{\Psi}\,^{k}_1(y^k)$. Hence $\overset{\circ}{\Psi}_1(y)\subset \liminf_k\overset{\circ}{\Psi}\,^{k}_1(y^k)$ and the result follows from Definition~\ref{def:g-convergence}.

Part~$(b)$ appears in~\cite[Remark~2.2.2]{LM92} and~$(c)$ follows similarly as in~$(a)$.
\end{proof}

\begin{remark} Taking stationary sequences, we infer the following continuity properties for these mappings: $(a)$
If $f$ is lsc and $g$ is usc, then $\Psi_1$ is osc; $(b)$ If
$C$ is closed, then $\Psi_2$ is osc; $(c)$ If $g_i$ is lsc for all $i=1,2,\ldots ,p$, then $\Psi_3$ is osc.
\end{remark}

Since feasible mappings of parametric problems can be compositions or Cartesian products of multifunctions, particularly the aforementioned ones, we establish the continuity properties of these operations. To do this, we use continuous and graphical convergence notions. Part~$(c)$ of Proposition~\ref{l1} appears in~\cite[Proposition~4.11$(a)$]{LS21} for $M=2$.

\begin{proposition}\label{l1}
Let $\Psi^k_1,\Psi_1\colon\R^\ell\tos \R^n$ and $\Psi^k_2,\Psi_2\colon\R^n\tos \R^m$ be given.
\begin{description}
    \item[$(a)$] If $g\mbox{-}\!\limsup\nolimits_k\Psi^k_1\subset\Psi_1$ with~$\{\Psi^k_1\}$ elb and $g\mbox{-}\!\limsup\nolimits_k\Psi^k_2\subset\Psi_2$, then
$$g\mbox{-}\!\limsup\nolimits_k(\Psi^k_2\circ\Psi^k_1)\subset(\Psi_2\circ\Psi_1).$$
    \item[$(b)$] If $\Psi_1\subset g\mbox{-}\!\liminf\nolimits_k\Psi^k_1$ and $\Psi_2 \subset c\mbox{-}\!\liminf\nolimits_k\Psi^k_2$, then
$$(\Psi_2\circ\Psi_1)\subset c\mbox{-}\!\liminf\nolimits_k(\Psi^k_2\circ\Psi^k_1).$$
    \item[$(c)$] If $\Psi^k_1\stackrel{g}{\to}\Psi_1$ with~$\{\Psi^k_1\}$ elb and~$\Psi^k_2\stackrel{c}{\to}\Psi_2$, then
$\Psi^k_2\circ\Psi^k_1 \stackrel{c}{\to}\Psi_2\circ\Psi_1$.
\end{description}
\end{proposition}

\begin{proof}
$(a)$ We prove $\limsup_ k \operatorname{gph}(\Psi^k_2\circ \Psi^k _1)\subset \operatorname{gph}(\Psi_2\circ \Psi_1)$. If $(x,z)$ is in the left-side set, then there exist
$z^{k _j}\in\Psi_2^{k_j}(y^{k_j})$, $y^{k_j}\in \Psi^{ k_j}_1(x^{k_j})$ and $\{x^{k_j}\}$ such that $x^{k_j} \to x$ and $z^{k_j} \to z$. By elb, we have $y^{k_j}\to y$ for some $y$, up to subsequences. Therefore $(y^{k_j},z^{k _j})\in\operatorname{gph}\Psi_2^{k_j}\to (y,z)$ and
$(x^{k_j},y^{k _j})\in\operatorname{gph}\Psi_1^{k_j}\to (x,y)$ that by hypothesis imply
$(y,z)\in\operatorname{gph}\Psi_2$ and $(x,y)\in\operatorname{gph}\Psi_1$. Hence $(x,z)\in\operatorname{gph}(\Psi_2\circ \Psi_1)$.

$(b)$ We prove $ (\Psi_2\circ \Psi_1)(x)\subset\liminf_ k (\Psi^k_2\circ \Psi^k _1)(x^k )$ for every $x^k\to x$. If $z\in(\Psi_2\circ \Psi_1)(x)$, then $z\in\Psi_2(y)$ for some $y\in \Psi_1(x)$. As $(x,y)\in \operatorname{gph}\Psi_1$, by hypothesis there exists
$(x^{k},y^{k})\in\operatorname{gph}\Psi_1^{k}\to (x,y)$. Thus, $y^k\in \Psi_1^{k}(x^k)$ for all $k$. On the other hand, by hypothesis for $y^k\to y$, we have $\Psi_2(y)\subset \liminf_k\Psi^k_2(y^k)$. Since $z\in\Psi_2(y)$, there exists $z^k\in \Psi^k_2(y^k)\to z$. Therefore $z^k\in(\Psi^k_2\circ\Psi_1^k)(x^k)\to z$ and we conclude that $z\in\liminf_k(\Psi^k_2\circ\Psi_1^k) (x^k)$. The result follows by~\eqref{eq:c-liminf-phi}.

$(c)$ It follows from $(a)$--$(b)$.
\end{proof}

\begin{remark}
Considering the unexplicit mappings $\Psi_1^k\equiv C_k$ and $\Psi_1\equiv C$ where $C_k,C$ are nonempty sets for all $k$, from part~$(c)$ of the previous result and Proposition~\ref{prop:ex-feas-mappings}$(b)$, we infer the following particular case of~\cite[Theorem~5.53$(c)$]{RW09}: `if $C_k\to C$ with $\{C_k\}$ eventually bounded and $\Psi^k_2\stackrel{c}{\to}\Psi_2$, then $\Psi_2(C_k)\to \Psi_2(C)$'. Notice that $C_k\to C$ with $\{C_k\}$ eventually bounded implies total convergence $C_k\stackrel{t}{\to} C$ by~\cite[Theorem~4.25$(d)$]{RW09} and $C^\infty=\{0\}$ that imply the hypothesis of the aforementioned result.
\end{remark}

\begin{proposition}\label{l2}
Let $\Phi_m^k,\Phi_m:\R^\ell\tos \R^n$ be given for $m\in\{1,\ldots,M\}$ and consider the mappings $\Psi^k (x):=\prod_{m=1}^M \Phi^k_m(x)$ and $\Psi(x):=\prod_{m=1}^M \Phi_m(x)$. Then
\begin{description}
\item[$(a)$] If $g\mbox{-}\!\limsup\nolimits_k\Phi^k_m\subset\Phi_m$ for all $m\in\{1,\ldots,M\}$, then $g\mbox{-}\!\limsup\nolimits_k\Psi^k\subset\Psi$.
\item[$(b)$] If $\Phi_1 \subset g\mbox{-}\!\liminf\nolimits_k\Phi^k_1$ and $\Phi_m\subset c\mbox{-}\!\liminf\nolimits_k\Phi^k_m$ for all $m\in\{2,\ldots ,M\}$, then $\Psi\subset g\mbox{-}\!\liminf\nolimits_k\Psi^k$.
\item[$(c)$] If $\Phi_1^k \stackrel{g}{\to}\Phi_1$ and $\Phi_m^k \stackrel{c}{\to}\Phi_m$ for all $m\in\{2,\ldots ,M\}$, then $\Psi^k \stackrel{c}{\to}\Psi$. In particular, if $\Phi_m^k \stackrel{c}{\to}\Phi_m$ for all $m\in\{1,\ldots,M\}$, then $\Psi^k\stackrel{c}{\to}\Psi$.
\end{description}
\end{proposition}

\begin{proof} $(a)$ We prove $\limsup_k\operatorname{gph}\Psi^k\subset \operatorname{gph}\Psi$. If $(x,y_1,\ldots ,y_M)$ is in the left-hand side, then there exists
$(x^{k_j},y_1^{k_j},\ldots ,y_M^{k_j})\in\operatorname{gph}\Psi^{k_j}\to (x,y_1,\ldots ,y_M)$. Therefore $(x^{k_j},y^{k_j}_m)\in\operatorname{gph}\Phi^{k_j}_m\to (x,y_m)$ for all $m\in\{1,\ldots, M\}$; thus, $(x,y_m)\in \limsup_k\operatorname{gph}\Phi_m^k$ that by hypothesis implies
$(x,y_m)\in\operatorname{gph}\Phi_m$ for such $m$. Hence $(x,y_1,\ldots ,y_M)\in\operatorname{gph}\Psi$.

$(b)$ We prove $\Psi(x)\subset\liminf_k\Psi^k(x^k)$ for all $x^k\to x$. If $(y_1,\ldots ,y_M)\in \Psi(x)$, then $y_m\in \Phi_m(x)$ for all $m\in\{1,\ldots, M\}$.
By hypothesis, there exists $(x^{k},y^{k}_1)\in\operatorname{gph}\Phi^{k}_1\to (x,y_1)$. Moreover, as $\Phi_m(x)\subset \liminf_k\Phi_m^k(x^k)$ for all $m\in\{2,\ldots, M\}$, there exist $y_m^k\in \Phi_m^k(x^k)\to y_m$ for such $m$. Therefore $(y_1^{k},\ldots ,y_M^{k})\in\Psi^{k}(x^k)\to (y_1,\ldots ,y_M)$ and $(y_1,\ldots ,y_M)$ is in the right-side set.

$(c)$ It follows from~$(a)$--$(b)$.
\end{proof}

\begin{remark}
Considering the unexplicit mappings $\Phi_m^k\equiv C_m^k$ and $\Phi_m\equiv C_m$ where $C_m^k,C_m$ are nonempty sets for every $k$ and $m\in\{1,\ldots,M\}$, from part~$(c)$ of the previous result and Proposition~\ref{prop:ex-feas-mappings}$(b)$, we infer \cite[Exercise~4.29$(a)$]{RW09}: `if $C_m^k\to C_m$ for all $m\in\{1,\ldots,M\}$, then $C_1^k\times\cdots\times C_M^k\to C_1\times\cdots\times C_M$'.
\end{remark}

The literature on the stability of marginal functions and solution mapping is mainly restricted to finite-valued objective functions and nonempty-valued feasible mappings. To extend such results, we define the following nonempty-valuedness and boundedness notions for sequences of multifunctions. The second and the third ones appear in~\cite{LS21,RW09}.

\begin{definition}
A sequence $\Psi^k\colon \R^\ell\tos \R^n$ is said to be:
\begin{itemize}
\item eventually nonempty-valued $($env$)$, if there exists $N\in\mathbb{N}$ such that $\Psi^k$ is nonempty-valued for every $k\ge N$.
\item eventually uniformly bounded $($eub$)$, if $\{\Psi^k(\R^\ell)\}$ is eventually bounded.
\item  eventually locally bounded
$($elb$)$, if $\{\Psi^k(y^k)\}$ is eventually bounded for every $y^k\to y$.
\item strongly eventually locally bounded $($selb$)$, if $\{\Psi^k(y^k)\}$ is strongly eventually bounded for every $y^k\to y$.
\end{itemize}
\end{definition}

A sequence of nonempty-valued multifunctions is env. An eub sequence is elb. An elb and env sequence of multifunctions is selb. A sequence $\{\Psi^k\}$ is elb iff for every~$y$ there exists a neighborhood $U$ of $y$ such that $\{\Psi^k(U)\}$ is eventually bounded.
When $\Psi^k\equiv\Psi$, then eub $($resp. elb$)$ property reduces to boundedness $($resp. local boundedness$)$ of $\Psi$.

\begin{remark}
Lignola and Morgan~\cite{LM92} defined the elb property as follows: `For any convergent sequence~$\{y^k\}$ and any sequence~$\{x^k\}$ with $x^k\in\Psi^k(y^k)$ for all~$k$, the sequence~$\{x^k\}$ has a convergent subsequence'.
\end{remark}

We now prove some properties of these notions.

\begin{proposition}\label{prop:elb-cont} Let $\Psi^k,\Psi\colon \R^\ell\tos \R^n$ be given. Then
\begin{description}
  \item[$(a)$] If $\{\Psi^k\}$ is elb and $\Psi \subset c\mbox{-}\!\liminf\nolimits_k\Psi^k$, then $\Psi$ is locally bounded.
  \item[$(b)$] If $\{\Psi^k\}$ is selb, then $\limsup\nolimits_k \Psi^k(y^k)$ is nonempty and compact for every $y^k\to~y$.
  \item[$(c)$] If $g\mbox{-}\!\limsup\nolimits_k\Psi^k\subset \Psi$ with $\{\Psi^k\}$ elb and $\Psi(y)=\emptyset$ for some $y$, then for every $y^k\to y$ there exists
      $N\in\mathbb{N}$
such that $\Psi^k(y^k)=\emptyset$ for all $k\ge N$.
 \item[$(d)$] If $\{\Psi^k\}$ is eub and $\Psi \subset g\mbox{-}\!\liminf\nolimits_k\Psi^k$, then $\Psi$ is bounded.
\end{description}
\end{proposition}

\begin{proof} $(a)$ For a fixed~$y$ there exist $r>0$, a neighborhood $U$ of $y$, and $N\in\mathbb{N}$ such that $\Psi^k(z)\subset r\mathbb{B}$ for all $k\ge N$ and $z\in U$.
By~\eqref{eq:c-liminf-phi}, we have $\Psi(z)\subset\liminf_k\Psi^k(z)$; thus, $\Psi(z)\subset  r\mathbb{B}$ for all $z\in U$. Hence $\Psi$ is locally bounded at $y$.

$(b)$ The result follows from Proposition~\ref{prop:p2}.

$(c)$ Let $\Psi(y)=\emptyset$ and $y^k\to y$. By elb there exists $r>0$ and $N_1\in\mathbb{N}$ such that $\bigcup_{k\ge N_1}\Psi^k(y^k)\subset r\mathbb{B}$. By~$\eqref{eq:g-limsup}$,
we have
$\limsup\nolimits_k\Psi^k(y^k)=\emptyset$. By the escaping to the horizon property, for such an $r$, there exists $N_2\in\mathbb{N}$ such that $\Psi^k(y^k)\cap r\mathbb{B}=\emptyset$
for all $k\ge N_2$, a contradiction if $\Psi^k(y^k)\ne\emptyset$ for some $k\ge N:=\max\{N_1,N_2\}$.

$(d)$ By hypothesis there exist $r>0$ and $N\in\mathbb{N}$ such that $\Psi^k(\R^m)\subset r\mathbb{B}$ for all $k\ge N$. Let $y$ be fixed. If $x\in\Psi(y)$, then by~\eqref{eq:g-liminf}
there exists $y^k\to y$ such that $x\in\liminf_k\Psi^k(y^k)$. By the above inclusion, we have $x\in  r\mathbb{B}$. Hence $\Psi(y)\subset  r\mathbb{B}$. The result follows since $y$
was arbitrary.
\end{proof}

We show that the boundedness properties of approximate feasible mappings are inherited by the approximate solution mappings.

\begin{proposition}\label{prop:elb-cont1} Let $\{u^k\}$, $\{\Phi^k\}$ and  $\{S^k\}$ be the approximations in~\eqref{eq:aprox-val-funct-sol-mult}. Then
\begin{description}
\item[$(a)$] If $\{\Phi^k\}$ is elb $($resp. eub$)$, then $\{S^k\}$ is elb $($resp. eub$)$.
\item[$(b)$] If $\{\Phi^k\}$ is selb with $\Phi^k$ closed-valued and $u^k(\cdot,y)$ usc for all $k$
      and $y$, then $\{S^k\}$ is selb.
\end{description}
\end{proposition}

\begin{proof} Part $(a)$ follows from $S^k\subset \Phi^k$ for all~$k$. To check part~$(b)$, we have that $\{S^k\}$ is elb by $(a)$. For any $y^k\to y$
there exists a subsequence $\{\Phi^{k_j}(y^{k_j})\}$ of nonempty sets that by hypothesis implies that $\{S^{k_j}(y^{k_j})\}$ are nonempty. So $\{S^k\}$ is selb.
\end{proof}

\section{Main results}\label{sec:3}
We study the behavior of the value function $v$ and the solution mapping $S$ of problem~$(\mathcal{P}_y)$ when the data $u$ and $\Phi$ are approximated as in~\eqref{eq:aprox-val-funct-sol-mult} by using the convergence notions in~Section~\ref{sec:2}. To do this, we establish the convergence properties of sequences of approximate value functions~$\{v^k\}$ and
solution mappings $\{S^k\}$.

First, we obtain a stability result that is a perturbed counterpart of Proposition~\ref{prop:berge-prop}$(a)$. Part~$(b)$ enhances~\cite[Proposition~4.2]{LM97}.

\begin{proposition}\label{prop:main0} Let $u\le e\mbox{-}\!\liminf\nolimits_k u^k$. Then
\begin{description}
  \item[$(a)$] If $\Phi \subset g\mbox{-}\!\liminf\nolimits_k\Phi^k$ with $\Phi$ compact-valued and $u(\cdot,y)$ is usc for all $y$, then $v\le h\mbox{-}\!\liminf\nolimits_k v^k$.
  \item[$(b)$] If $\Phi \subset c\mbox{-}\!\liminf\nolimits_k\Phi^k$, then $v\le e\mbox{-}\!\liminf\nolimits_k v^k$.
\end{description}
\end{proposition}

\begin{proof} $(a)$ If $y\notin\mathrm{dom}\,\Phi$, then $v(y)=-\infty$ and $\liminf_kv^k(y^k)\ge v(y)$ for all $y^k\to y$. On the other hand, if $y\in\mathrm{dom}\,\Phi$, then by
hypothesis there exists $x\in\Phi(y)$ such that $v(y)=u(x,y)$. As $(y,x)\in\mathrm{gph}\,\Phi$,  we have $(y,x)\in\liminf_k \mathrm{gph}\,\Phi^k$ and there exists $(y^k,x^k)\to (y,x)$
such that $x^k\in\Phi^k(y^k)$ for all~$k$. As $v^k(y^k)\ge u^k(x^k,y^k)$ for all~$k$, after taking the liminf,
we obtain $\liminf_k v^k(y^k)\geq u(x,y)$. Hence $\liminf_k v^k(y^k)\geq v(y)$. The result follows by~\eqref{eq:h-liminf}.

$(b)$ Let $y^k\to y$. If $y\notin\mathrm{dom}\,\Phi$, then $v(y)=-\infty$ and $\liminf_kv^k(y^k)\ge v(y)$ holds for all $y^k\to y$. On the other hand, if $y\in\mathrm{dom}\,\Phi$, then
by~\eqref{eq:c-liminf-phi} for any $x\in\Phi(y)$ there exists $x^k\in \Phi^k(y^k)\to x$. As $v^k(y^k)\ge u^k(x^k,y^k)$ for all $k$, after taking the liminf, we obtain
$\liminf\nolimits_k
v^k(y^k)\geq u(x,y)$. From this and since $x\in\Phi(y)$ was arbitrary, we obtain $\liminf\nolimits_k v^k(y^k)\geq v(y)$. The result follows by~\eqref{eq:e-liminf}.
\end{proof}

Under a compactness assumption, it is easy to see that \cite[Proposition 4.2]{LM97} can be deduced from Proposition~\ref{prop:main0}$(a)$. However, the converse does not generally hold, as the following example shows.

\begin{example}
Consider $\Phi^k$, $\Phi$, $u^k$ and $u$ defined as
$\Phi^k(y)=\Phi(y)\equiv[0,1]$, $u^k(x,y)=y/k$, and $u(x,y)=0$ if $y>0$ and
$u(x,y)=-1$ otherwise. Clearly, $u\leq e\mbox{-}\!\liminf\nolimits_k u^k$ with $u(\cdot,y)$ usc and $\Phi\subset g\mbox{-}\!\liminf\nolimits_k \Phi^k$ with $\Phi$ compact-valued. However, it holds that
$0=\limsup_k u(0,1/k)\not<u(0,0)=-1$.
That means that, $v\le h\mbox{-}\!\liminf\nolimits_k v^k$ is guaranteed by Proposition~\ref{prop:main0}$(a)$ and not by \cite[Proposition 4.2]{LM97}.
\end{example}

The next example shows that we cannot relax the inclusions in Proposition~\ref{prop:main0}.

\begin{example}
Consider  $\Phi^k$, $\Phi$, $u^k$ and $u$ defined as
$\Phi^k(y)\equiv [0,1]$, $\Phi(y)\equiv[0,2]$, and $u^k(x,y)=x$.
Clearly, the hypothesis of Proposition~\ref{prop:main0}$(a)$ holds but $\Phi\not \subset g\mbox{-}\!\liminf\nolimits_k\Phi^k$ (so $\Phi\not \subset c\mbox{-}\!\liminf\nolimits_k\Phi^k$). As $v^k=1$ for all $k$ but $v=2$, we have $v\not\le h\mbox{-}\!\liminf\nolimits_k v^k$ and $v\not\le e\mbox{-}\!\liminf\nolimits_k v^k$.
\end{example}

Now, we obtain a stability result that is a perturbed counterpart of Proposition~\ref{prop:berge-prop}$(b)$. This result extends~\cite[Proposition~4.1]{LM97} to problems with objective functions and feasible mappings having non compact domains.

\begin{proposition}\label{prop:main1} If $h\mbox{-}\!\limsup\nolimits_k u^k\le u$ and $g\mbox{-}\!\limsup\nolimits_k\Phi^k\subset \Phi$ with $\{\Phi^k\}$ elb, then $h\mbox{-}\!\limsup\nolimits_k v^k\le v$.
\end{proposition}

\begin{proof} We follow the line of reasoning in~\cite{LM97}.
On the contrary, suppose that there exist $\tilde y\in\R^m$, $y^k\to \tilde y$ and $a>0$ such that $\limsup_k v^k(y^k)>a>v(\tilde y)$. By the first inequality there exists $N$ such that for every $k$ there exists $n_k\ge k$ such that $v^{n_k}(y^{n_k})>a$. By definition of suprema, there exists $x^{n_k}\in \Phi^{n_k}(y^{n_k})$ such that $u^{n_k}(x^{n_k},y^{n_k})>a$ for such $k$.
By elb, we have $x^{n_k}\to \tilde x$ for some $\tilde x$, up to subsequences.
As $(y^{n_k},x^{n_k})\in\mathrm{gph}\,\Phi^{n_k}\to (\tilde y,\tilde x)$, we have
$(\tilde y,\tilde x)\in\limsup_k\mathrm{gph}\,\Phi^k$ that implies
$\tilde x\in\Phi(\tilde y)$ by hypothesis. Moreover, we have $v(\tilde y)\ge u(\tilde x,\tilde y)\ge \limsup_k u^{n_k}(x^{n_k},y^{n_k})\ge a$ by hypothesis, a contradiction.
\end{proof}

We obtain stability properties of the value function and of the solution mapping under continuous convergence of the data.

\begin{theorem}\label{MT0}Let $u^k\stackrel{c}{\to}u$ and $\Phi^k\stackrel{c}{\to}\Phi$. Then
\begin{description}
  \item[$(a)$] $g\mbox{-}\!\limsup\nolimits_k S^k\subset S$.
  \item[$(b)$] If $\{\Phi^k\}$ is elb, then $v^k\stackrel{c}{\to} v$ and $\{S^k\}$ is elb.
  \item[$(c)$] If $\{\Phi^k\}$ is selb, $\Phi^k$ is closed-valued and $u^k(\cdot,y)$ is usc for all $k$ and $y$, then~$\{S^k\}$ is selb, $g\mbox{-}\!\limsup\nolimits_k S^k$ is
      nonempty-bounded-valued and $S$ is nonempty-valued.
\end{description}
\end{theorem}

\begin{proof} $(a)$ Let $y^k\to y$ be fixed. If $x\in\limsup\nolimits_k S^{k}(y^k)$, then there exists
$x^{k_j}\in S^{k_j}(y^{k_j})\to x$. Hence $x^{k_j}\in \Phi^{k_j}(y^{k_j})$ and $v^{k_j}(y^{k_j})=u^{k_j}(x^{k_j},y^{k_j})$ for all~$j$. As $x\in\limsup_k\Phi^k(y^k)$,
by~\eqref{eq:g-limsup} we have $x\in\Phi(y)$. After taking the limit to the last equality, using Proposition~\ref{prop:main0}$(b)$ and
Proposition~\ref{prop:c-convergence-subseq}, we obtain
$$v(y)\le\liminf\nolimits_k v^{k}(y^{k})\le\liminf\nolimits_j v^{k_j}(y^{k_j})=\lim\nolimits_ju^{k_j}(x^{k_j},y^{k_j})=u(x,y);$$
i.e., $v(y)= u(x,y)$ and thus $x\in S(y)$. Hence $\limsup\nolimits_k S^{k}(y^k)\subset S(y)$ and the inclusion follows from~\eqref{eq:g-limsup}.

$(b)$ The first part follows from Propositions~\ref{prop:main0}$(b)$ and~\ref{prop:main1}. The second part follows from Proposition~\ref{prop:elb-cont1}$(a)$.

$(c)$ The first part follows from Proposition~\ref{prop:elb-cont1}$(b)$. The remaining  follows from the first part, Proposition~\ref{prop:p2}, and~$(a)$.
\end{proof}

\begin{remark}\label{ex:1}
Concerning Theorem~\ref{MT0}, few remarks are needed.
\begin{enumerate}
\item In part $(b)$,  assumption elb cannot be dropped. Indeed, for
$u^k(x,y)=(1/k)|x|$, $u(x,y)\equiv 0$, and $\Phi^k(y)=\Phi(y)\equiv\R$,
we have $u^k\stackrel{c}{\to}u$, $\Phi^k\stackrel{c}{\to}\Phi$, $\{\Phi^k\}$ is not elb,
$S^k(y)\equiv\emptyset$, $v^k(y)\equiv +\infty$, $v(y)\equiv 0$, and $S(y)\equiv\R$. So~$(b)$--$(c)$ fail to hold. Clearly, $S\not\subset
c\mbox{-}\!\liminf\nolimits_k S^k$; thus, $S^k\not\stackrel{c}{\to}S$.
\begin{figure}[h!]
    \centering
    \begin{tikzpicture}[>=latex,scale=1.9]
    \draw[gray,-](-1.0,0)--(1.0,0);
    \draw[gray,->](0,-0.25)--(0,1.25);
    \draw[thick,-](-1,1)--(0,0)--(1,1)node at(1,1)[right]{$u_1$};
    \draw[thick,-](-1,0.5)--(0,0)--(1,0.5)node at (1,0.5)[right]{$u_2$};
    \draw[thick,-](-1,0.25)--(0,0)--(1,0.25) node at (1,0.25)[right]{$u_k$};
    \draw[-,thick](-1,0)--(1,0)node at (1,0)[below right]{$u$};
    \draw(0,-0.05)node[below ]{$\Phi^k=\Phi$};
    \draw(0.97,0.2)node[]{$\vdots$};
    \draw(-0.97,0.2)node[]{$\vdots$};
    \end{tikzpicture}
\caption{Assumption elb cannot be dropped}\label{fig:1}
\end{figure}
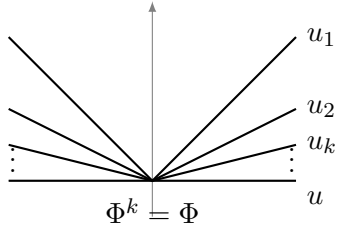

\item Continuous convergence of objective functions cannot be replaced by pointwise or uniform convergence (see~\cite{RW09} for definitions). Indeed, for
\[
u^k(x,y)=u(x,y)=
\left\{
  \begin{array}{ll}
    1, & \hbox{if $x\in\qq$;} \\
    0, & \hbox{elsewhere.}
  \end{array}
\right.
\;\;\mbox{ and }\;\; \Phi^k(y)=\Phi(y)\equiv[0,1],
\]
we have $u^k\stackrel{p}{\to}u$, $u^k\stackrel{u}{\to}u$, $u^k\not\stackrel{c}{\to}u$, $\Phi^k\stackrel{c}{\to}\Phi$, $v^k(y)= v(y)\equiv 1$, $S^k(y)= S(y)\equiv [0,1]\cap\qq$ and
$\lim_k S^k(y^k)=[0,1]$ for every $y^k\to y$. So~$(a)$ fails to hold.

\item Continuous convergence of feasible mappings cannot be replaced by uniform convergence $($and so by pointwise convergence$)$. Indeed, for
\[
u^k(x,y)=u(x,y)\equiv 1\;\,\mbox{ and }\;\,\Phi^k(y)=\Phi(y)=\begin{cases}
\{0\},& \hbox{if $y\in\mathbb{Q}$;}\\
\{1\},&\hbox{elsewhere,}
\end{cases}
\]
we have $u^k\stackrel{c}{\to}u$, $\Phi^k\stackrel{u}{\to}\Phi$, $\Phi^k\not\!\stackrel{c}{\to}\Phi$ and
\[
S^k(y)=S(y)=\begin{cases}
\{0\},& \hbox{if $y\in\mathbb{Q}$;}\\
\{1\},&\hbox{elsewhere}.
\end{cases}
\]
As $\limsup_k S^k(y^k)=\{0,1\}$ for every $y^k\to y$, we have $g\mbox{-}\!\limsup\nolimits_k S^k\not\subset S$. So~$(a)$ fails to hold.

\item The inclusion in part $(a)$ could be strict. Indeed, for $u^k(x,y)=u(x,y)=f(x)$, $\Phi^k(y)=[0,3+1/k]$, and $\Phi(y)\equiv[0,3]$ where $f\colon\R\to\R$ is defined by
\[
f(x)=\left\{
         \begin{array}{ll}
           x, & \hbox{if $x<1$;} \\
           2-x, & \hbox{if $1\le x<2$;} \\
           x-2, & \hbox{if $2\le x$,}
         \end{array}
       \right.
\]
we have $u^k\stackrel{c}{\to}u$, $\Phi^k\stackrel{c}{\to}\Phi$, $S^k(y)\equiv\{3+1/k\}$ and $S(y)\equiv \{1,3\}$. Thus, $\limsup_{k}S^k(y^k)\subsetneq S(y)$ for every $y^k\to y$, see Figure~\ref{fig:4}.

\begin{figure}[h!]
    \centering
  \begin{tikzpicture}[>=latex,scale=1.2]
  \draw[->,gray](-1,0)--(3.7,0);
  \draw[->,gray](0,-0.5)--(0,1.75);
  \draw[thick,-](-0.5,-0.5)--(1,1)--(2,0)--(3.5,1.5)node[right]{$f$};
  \draw[thick](0,-0.1)--(3.2,-0.1)node[below right]{$\Phi^k(y)$};
  \fill[](3.2,-0.1)circle(1.1pt) node at(3.2,0.2)[right]{$S^k(y)$};
  \draw[thick](0,-0.05)--(3,-0.05)node[below ]{$\Phi(y)$};
  \fill[](1,-0.05)circle(1.1pt);
\fill[](3,-0.05)circle(1.1pt)node at(3,0.2)[]{$S(y)$};
  \end{tikzpicture}
    \caption{Strict inclusion in part $(a)$}
    \label{fig:4}
\end{figure}
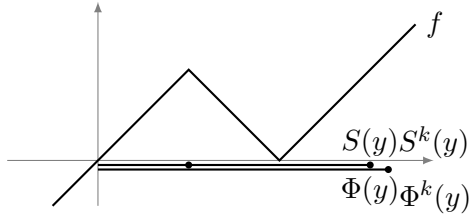

Even if $u$ and $u^k$ are concave the inclusion, in part $(a)$, could be strict. Indeed, for
$u^k(x,y)=\dfrac{1}{k}f(x)$, $u(x,y)\equiv 0$, and $\Phi^k(y)=\Phi(y)\equiv[0,3]$ where $f:\R\to\R$ is defined by
\[
f(x)=
\left\{
  \begin{array}{ll}
    x, & \hbox{if $x<1$;} \\
    1, & \hbox{if $1\leq x<2$;} \\
    3-x, & \hbox{if $2\leq x$,}
  \end{array}
\right.
\]
we have $u^k\stackrel{c}{\to}u$, $\Phi^k\stackrel{c}{\to}\Phi$, $S^k(y)\equiv [1,2]$ and $S(y)\equiv [0,3]$, see Figure~\ref{fig:5}. Consequently
$ \limsup_k S^k(y^k)\subsetneq S(y)$, for every $y^k\to y$.

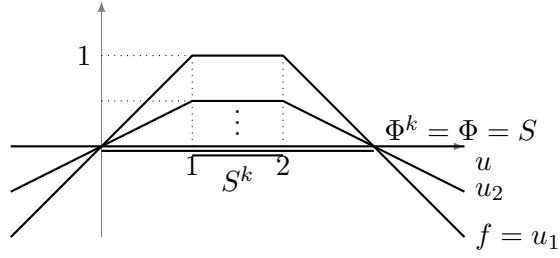
\begin{figure}[H]
    \centering
  \begin{tikzpicture}[>=latex,scale=1.2]
  \draw[->,gray](-1,0)--(4,0);
  \draw[->,gray](0,-1)--(0,1.6);
  \draw[thick,-](-1,-1)--(1,1)--(2,1)--(4,-1)node[right]{$f=u_1$};
    \draw[thick,-](-1,-0.5)--(1,0.5)--(2,0.5)--(4,-0.5)node[right]{$u_2$};
  \draw[dotted](1,0)--(1,1);
  \draw[dotted](2,0)--(2,1);
  \draw[dotted](0,1)--(1,1);
  \draw[dotted](0,0.5)--(1,0.5);
  \draw(0,1)node[left]{$1$};
  \draw(1.5,0.35)node[]{$\vdots$};
  \draw[thick,-](-1,0)--(4,0)node[below right]{$u$};
\draw[thick,-](0,-0.05)--(3,-0.05)node[above right]{$\Phi^k=\Phi=S$};
\draw[thick,-](1,-0.1)--(2,-0.1)node at (1.5,-0.1) [below]{$S^k$};
  \draw(1,0)node[below]{$1$};
  \draw(2,0)node[below]{$2$};
  \end{tikzpicture}
    \caption{Strict inclusion in part $(a)$ even if concavity is assumed}
    \label{fig:5}
\end{figure}
\end{enumerate}
\end{remark}

\begin{remark}
We point out that inclusion $S\subset c\mbox{-}\!\liminf\nolimits_k S^k$ does not hold in Theorem~\ref{MT0} even for continuous mappings and under continuous convergence. Indeed, for $u^k(x,y)=\dfrac{1}{k}|x|$,  $u(x,y)\equiv 0$
and
$\Phi^k(y)=\Phi(y)\equiv[-1,1]$, we have $u^k\stackrel{c}{\to}u$, $\Phi^k\stackrel{c}{\to}\Phi$, $v^k(y)\equiv 1/k$, $v(y)\equiv 0$, $S^k(y)\equiv\{-1,1\}$ and $S(y)\equiv[-1,1]$.
Thus,
$S(y)\not\subset \liminf_k S^k(y^k)$ for every $y^k\to y$.
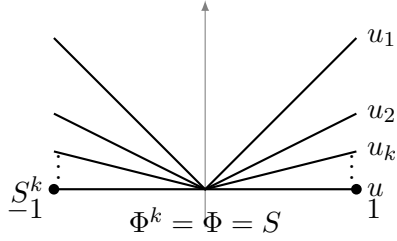
\begin{figure}[h!]
    \centering
    \begin{tikzpicture}[>=latex,scale=2]
    \draw[gray,-](-1.0,0)--(1.0,0);
    \draw[gray,->](0,-0.25)--(0,1.25);
    \draw[thick,-](-1,1)--(0,0)--(1,1)node at(1,1)[right]{$u_1$};
    \draw[thick,-](-1,0.5)--(0,0)--(1,0.5)node at (1,0.5)[right]{$u_2$};
    \draw[thick,-](-1,0.25)--(0,0)--(1,0.25) node at (1,0.25)[right]{$u_k$};
    \draw[-,thick](-1,0)--(1,0)node at (1,0)[right]{$u$};
    \draw(0,-0.05)node[below ]{$\Phi^k=\Phi=S$};
    \fill[](-1,0)circle(1pt);
    \fill[](1,0)circle(1pt)node at (-1,0)[left]{$S^k$};
    \draw(0.97,0.2)node[]{$\vdots$};
    \draw(-0.97,0.2)node[]{$\vdots$};
    \draw(-1,0)node[below left]{$-1$};
    \draw(1,0)node[below right]{$1$};
    \end{tikzpicture}
\caption{The inclusion $S\subset c\mbox{-}\!\liminf\nolimits_k S^k$ does not hold in general.}\label{fig:6}
\end{figure}
\end{remark}

By setting $u^k\equiv u$ and~$\Phi^k\equiv\Phi$ in Propositions~\ref{prop:main0}$(b)$ and~\ref{prop:main1}, and Theorem~\ref{MT0}, we infer the following result.

\begin{corollary} \label{cor:berge-variant}
\begin{description}
  \item[$(a)$] If $u\colon\R^{n+m}\to\overline{\R}$ is lsc and $\Phi\colon \R^m\tos \R^n$ is isc, then $v\colon\R^{m}\to\overline{\R}$ is lsc.
  \item[$(b)$] If $u\colon\R^{n+m}\to\overline{\R}$ is usc and $\Phi\colon \R^m\tos \R^n$ is osc, locally bounded, then $v\colon\R^{m}\to\overline{\R}$
      is usc.
  \item[$(c)$] If $u\colon\R^{n+m}\to \overline{\R}$ is continuous, $\Phi\colon \R^m\tos \R^n$ is continuous and locally bounded, then $v\colon\R^{m}\to
      \overline{\R}$ is
      continuous and $S\colon \R^m\tos \R^n$ is nonempty-valued, locally bounded and osc.
\end{description}
\end{corollary}

\begin{remark} Since $\Phi$ is usc and compact-valued iff $\Phi$ is osc and locally bounded; and $\Phi$ is lsc iff $\Phi$ is isc (see~\cite{RW09}), this result extends Proposition~\ref{prop:berge-prop} as follows:
\begin{enumerate}
\item Parts~$(a)$--$(b)$ hold for proper feasible mappings and not only for nonempty-valued mappings as in Proposition~\ref{prop:berge-prop}$(a)$,$(b)$. Indeed, we consider part~$(a)$. Let $y^k\to y$ be fixed. If $\Phi(y)=\emptyset$, then $v(y)=-\infty\le \liminf_k v(y^k)$ and $v$ is lsc at $y$. If $\Phi(y)\ne\emptyset$, then for every
$\varepsilon>0$ there exists $x_\varepsilon\in\Phi(y)$ such that $v(y)-\varepsilon<u(x_\varepsilon,y)$. As $\Phi$ is isc, we have $x_\varepsilon\in\liminf_k\Phi(y^k)$ and there exists $x^k\in\Phi(y^k)\to x_\varepsilon$. By hypothesis,
$v(y)-\varepsilon<u(x_\varepsilon,y)\le\liminf\nolimits_k u(x^k,y^k)\le\liminf\nolimits_k v(y^k)$.
As $\varepsilon>0$ was arbitrary, we have $v(y)\le \liminf_k v(y^k)$ and $v$ is lsc at $y$.

\item Part~$(c)$ holds for extended-valued functions and not only for finite functions as in Proposition~\ref{prop:berge-prop}$(c)$. Indeed, for $u(x,y)=f(x)$ where $f$ is the extended-valued continuous function $f(x)=-\infty$ if $x\le -\pi/2$, $f(x)=\tan x$ if $|x|<\pi/2$, $f(x)=+\infty$ if $x\ge\pi/2$ $($see~\cite[Example~3.3.1]{TW20}$)$, and for $\Phi(y)=\{y\}$ the hypothesis of $(c)$ holds. As $v(y)=f(y)$ and $S(y)=\{y\}$, the conclusions of part~$(c)$ hold.

\item The solution mapping $S$ may not be isc (so, neither continuous) under the hypotheses of part $(c)$. Indeed, for the continuous function $u(x,y)=|xy|$ and the continuous, locally bounded mapping $\Phi(y)\equiv[-1,1]$ for all $y\in\R$, we have $v(y)=|y|$ for all $y\in\R$. Moreover, $S(y)=\{-1,1\}$ if $y\ne 0$ and $S(0)=[-1,1]$. Clearly, for $y^k=1/k\to 0$ one has $S(0)\not\subset \liminf_k S(y^k)$.
\end{enumerate}
\end{remark}

We derive stability properties of the value function and the solution mapping under continuous convergence of the objective functions and graphical convergence of the feasible
mappings.

\begin{theorem}\label{MT1} Let $u^k\stackrel{c}{\to}u$ and $\Phi^k\stackrel{g}{\to}\Phi$. Then
\begin{description}
  \item[$(a)$]  If $\{\Phi^k\}$ is eub, then $\{S^k\}$ is eub and $S$ is bounded, compact-valued.
  \item[$(b)$] Let $\{\Phi^k\}$ be elb. Then
\begin{description}
  \item[$(i)$] $v^k\stackrel{h}{\to} v$ and $\{S^k\}$ is elb.
  \item[$(ii)$] $\forall y\in\R^m$, $\exists y^k\to y;$ $v^k(y^k)\to v(y)$ and $\limsup\nolimits_k S^k(y^k)\subset S(y)$.
\end{description}
\end{description}
\end{theorem}

\begin{proof} By Remarks~\ref{rem:h-conv}$(1)$ and~\ref{rem:continuous2}$(2)$, we infer that $u$ is continuous and $\Phi$ is osc.

$(a)$ The mapping $S$ is bounded since $S\subset \Phi$ and $\Phi$ is bounded by Proposition~\ref{prop:elb-cont}$(d)$. That $\{S^k\}$ is eub follows from Proposition~\ref{prop:elb-cont1}$(a)$. The last part follows from the first part and the
closed-valuedness of $S$ since $u$
is usc and $\Phi$ is closed-valued.

$(b)$, $(i)$ The first part follows from Propositions~\ref{prop:main0}$(a)$ and~\ref{prop:main1}. The second part follows from Proposition~\ref{prop:elb-cont1}$(a)$.

$(ii)$ By~$(i)$ and Remark~\ref{rem:h-conv}$(1)$ for every $y$ there exists $y^k\to y$ such that $v^k(y^k)\to v(y)$. If $x\in\limsup\nolimits_k S^k(y^k)$, then there exists
$x^{k_j}\in S^{k_j}(y^{k_j})\to x$. As $x^{k_j}\in \Phi^{k_j}(y^{k_j})\to x$, we have $x\in\limsup\nolimits_k \Phi^k(y^k)$ that
by~\eqref{eq:g-limsup} implies $x\in\Phi(y)$. After taking the liminf to $v^{k_j}(y^{k_j})=u^{k_j}(x^{k_j},y^{k_j})$, we obtain $v(y)\le u(x,y)$; i.e., $v(y)=u(x,y)$ and $x\in S(y)$.
Hence $\limsup\nolimits_k S^k(y^k)\subset S(y)$.
\end{proof}

Based on the above result, we derive a different type of stability result for the solution mapping.

\begin{corollary}\label{cor:stability-c-g}
Let $u^k\stackrel{c}{\to}u$ and $\Phi^k\stackrel{g}{\to}\Phi$.
If $\{\Phi^k\}$ is elb and env with $\Phi^k$ closed-valued and $u^k(\cdot,y)$ usc for all $k$ and $y$, then $\limsup\nolimits_k S^k(y^k)$ is nonempty and compact for every $y^k\to y$, and $(g\mbox{-}\!\limsup\nolimits_k S^k)(y)\cap S(y)\ne\emptyset$ for every $y$.
\end{corollary}

\begin{proof}
As $\{\Phi^k\}$ is selb, $\{S^k\}$ is selb by Proposition~\ref{prop:elb-cont1}$(b)$. This and Proposition~\ref{prop:elb-cont}$(b)$ imply that $\limsup\nolimits_k S^k(y^k)$ is nonempty compact for every $y^k\to y$. The last part follows from this and the pointwise formula for the graphical outer limit.
\end{proof}

\begin{remark}\label{ex:2}
A few remarks are needed regarding Theorem~\ref{MT1}.
\begin{enumerate}
\item The continuous convergence of objective functions cannot be replaced by pointwise or uniform convergence. This is shown in
Remark~\ref{ex:1}$(2)$ where $\Phi^k\stackrel{g}{\to}\Phi$.

\item Let $u^k(x,y)=u(x,y)=x$, $\Phi^k$, $\widetilde{\Phi}$ and $\widehat{\Phi}$  be the following multifunctions (see~\cite[p.~170]{RW09})
$$\Phi^k(y)=
	\left\{
	\begin{array}{ll}
		[0,1], & \hbox{if $0\le y\le 1-1/k$;} \\
		\displaystyle{[0, 2ky-2k+3]}, & \hbox{if $1-1/k\le y\le 1-1/(2k)$;} \\
		\displaystyle{[0, -2ky+2k+1]}, & \hbox{if $1-1/(2k)\le y\le 1$;}\\
		\emptyset, & \hbox{elsewhere,}
	\end{array}
	\right.$$
$$\widetilde{\Phi}(y)=\left\{
	\begin{array}{ll}
		[0,1], & \hbox{if $0\le y<1$;} \\
		\displaystyle{[0,2]}, & \hbox{if $y=1$;} \\
		\emptyset, & \hbox{elsewhere,}
	\end{array}
	\right.\;\;\mbox{ and }\;\; \widehat{\Phi}(y)=
	\left\{
	\begin{array}{ll}
		[0,1], & \hbox{if $0\le y\le 1$;} \\
		\emptyset, & \hbox{elsewhere.}
	\end{array}
	\right.$$
Their graphs are given in the following figure.
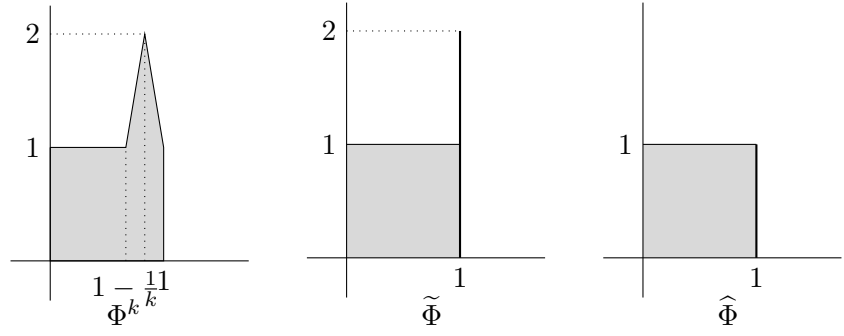
\begin{figure}[H]
    \centering
\begin{tikzpicture}[>=latex,scale=1.5]
\draw[fill=gray!30,samples=100](0,1)--(2/3,1)--(5/6,2)--(1,1)--(1,0)--(0,0)--(0,1);
\draw[-](-0.35,0)--(1.75,0);
\draw[-](0,-0.35)--(0,2.25);
\draw[dotted](2/3,0)--(2/3,1);
\draw[dotted](5/6,0)--(5/6,2);
\draw(2/3,0)node[below]{$1-\frac{1}{k}$};
\draw[dotted](0,2)--(5/6,2);
\draw(0,2)node[left]{$2$};
\draw(0,1)node[left]{$1$};
\draw(1,0)node[below]{$1$};
\draw(0.65,-0.25)node[below]{$\Phi^k$};
\end{tikzpicture}
\hspace{0.5cm}
\begin{tikzpicture}[>=latex,scale=1.5]
\fill[gray!30,samples=100](0,1)--(1,1)--(1,0)--(0,0)--(0,1);
\draw[-](-0.35,0)--(1.75,0);
\draw[-](0,-0.35)--(0,2.25);
\draw[](0,1)--(1,1);
\draw[thick](1,0)--(1,2);
\draw(0,1)node[left]{$1$};
\draw(1,0)node[below]{$1$};
\draw[dotted](0,2)--(1,2);
\draw(0,2)node[left]{$2$};
\draw(0.75,-0.25)node[below]{$\widetilde{\Phi}$};
\end{tikzpicture}
\hspace{0.5cm}
\begin{tikzpicture}[>=latex,scale=1.5]
\fill[gray!30,samples=100](0,1)--(1,1)--(1,0)--(0,0)--(0,1);
\draw[-](-0.35,0)--(1.75,0);
\draw[-](0,-0.35)--(0,2.25);
\draw[](0,1)--(1,1);
\draw[thick](1,0)--(1,1);
\draw(0,1)node[left]{$1$};
\draw(1,0)node[below]{$1$};
\draw(0.75,-0.25)node[below]{$\widehat{\Phi}$};
\end{tikzpicture}
\end{figure}

We observe that $\Phi^k\stackrel{g}{\to}\widetilde{\Phi}$, $\Phi^k\stackrel{p}{\to}\widehat{\Phi}$  (according to~\cite[Definition 5.31]{RW09}),
$u^k\stackrel{c}{\to}u$, $\{\Phi^k\}$ is elb,
$$v^k(y)=\left\{
  \begin{array}{ll}
    1, & \hbox{if $0\le y\le 1-\frac{1}{k}$;} \\
    2ky-2k+3, & \hbox{if $1-\frac{1}{k}\le y\le 1-\frac{1}{2k}$;} \\
    -2ky+2k+1, & \hbox{if $1-\frac{1}{2k}\le y\le 1$;}\\
+\infty, & \hbox{elsewhere;}
  \end{array}
\right.$$
$$S^k(y)=\left\{
  \begin{array}{ll}
    \{1\}, & \hbox{if $0\le y\le 1-\frac{1}{k}$;} \\
    \{2ky-2k+3\}, & \hbox{if $1-\frac{1}{k}\le y\le 1-\frac{1}{2k}$;} \\
    \{-2ky+2k+1\}, & \hbox{if $1-\frac{1}{2k}\le y\le 1$;}\\
\emptyset, & \hbox{elsewhere.}
  \end{array}
\right.$$

Under graphical or pointwise convergence of $\{\Phi^k\}$, we observe different convergence properties of the sequences $\{v^k\}$ and
 $\{S^k\}$:

\noindent $(i)$ If $\Phi^k\stackrel{g}{\to}\widetilde{\Phi}$, then $v^k\stackrel{h}{\to} \widetilde{v}$ and $g\mbox{-}\!\limsup\nolimits_k S^k\not\subset \widetilde{S}$ where
\[
\widetilde{v}(y)=\left\{
  \begin{array}{ll}
    1, & \hbox{if $0\le y<1$;} \\
   2, & \hbox{if $y=1$;} \\
    +\infty, & \hbox{elsewhere,}
  \end{array}
\right.\;\,\mbox{ and }\;\, \widetilde{S}(y)=\left\{
  \begin{array}{ll}
    \{1\}, & \hbox{if $0\le y<1$;} \\
   \{2\}, & \hbox{if $y=1$;} \\
    \emptyset, & \hbox{elsewhere.}
  \end{array}
\right.
\]

Indeed, the latter holds since for $y=1$ and $y^k=1-1/k$, we have
$$
\limsup\nolimits_k S^k(y^k)=\{1\}\not\subset \widetilde{S}(y)=\{2\}.
$$
So the counterpart of Theorem~\ref{MT0}$(a)$, that is expected for graphical convergence, does not hold.

\noindent $(ii)$ If $\Phi^k\stackrel{p}{\to}\widehat{\Phi}$; i.e., $\Phi^k(x)\to\widehat{\Phi}(x)$ for all $x$, then $v^k\not\stackrel{h}{\to}\widehat{v}$ and $g\mbox{-}\!\limsup\nolimits_k S^k\not\subset \widehat{S}$ where
\[
\widehat{v}(y)=
\left\{
  \begin{array}{ll}
    1, & \hbox{if $0\le y\le 1$;} \\
    +\infty, & \hbox{elsewhere,}
  \end{array}
\right.\;\,\mbox{ and }\;\,
\widehat{S}(y)=
\left\{
  \begin{array}{ll}
    \{1\}, & \hbox{if $0\le y\le 1$;} \\
    \emptyset, & \hbox{elsewhere.}
  \end{array}
\right.
\]

Indeed, for $y=1$ and $y^k=1-1/(2k)$, we have $\limsup_kv^k(y^k)=2\not\le\widehat{v}(y)=1$ and $\limsup\nolimits_k S^k(y^k)=\{2\}\not\subset \widehat{S}(y)=\{1\}$.

The second instance shows that graphical convergence of feasible mappings
cannot be replaced by its pointwise convergence.

Below we present the geometric representation of the functions $v^k$, $\widetilde{v}$ and $\widehat{v}$.
\begin{figure}[H]
    \centering
\begin{tikzpicture}[>=latex,scale=1.65]
\draw[samples=100](0,1)--(2/3,1)--(5/6,2)--(1,1);
\draw[-](-0.35,0)--(1.75,0);
\draw[-](0,-0.35)--(0,2.25);
\draw[dotted](2/3,0)--(2/3,1);
\draw[dotted](5/6,0)--(5/6,2);
\draw[dotted](1,0)--(1,1);
\fill[](1,1)circle(1pt);
\fill[](0,1)circle(1pt);
\draw(2/3,0)node[below]{$1-\frac{1}{k}$};
\draw[dotted](0,2)--(5/6,2);
\draw(0,2)node[left]{$2$};
\draw(0,1)node[left]{$1$};
\draw(1,0)node[below]{$1$};
\draw(0.75,-0.25)node[below]{$v^k$};
\end{tikzpicture}
\hspace{0.5cm}
\begin{tikzpicture}[>=latex,scale=1.65]
\draw[-](-0.35,0)--(1.75,0);
\draw[-](0,-0.35)--(0,2.25);
\draw[dotted](1,0)--(1,2);
\draw[samples=100](0,1)--(1,1);
\fill[](0,1)circle(1pt);
\fill[](1,2)circle(1pt);
\draw[](1,1)circle(1pt);
\fill[white](1,1)circle(1pt);
\draw(0,1)node[left]{$1$};
\draw(1,0)node[below]{$1$};
\draw[dotted](0,2)--(1,2);
\draw(0,2)node[left]{$2$};
\draw(0.75,-0.25)node[below]{$\widetilde{v}$};
\end{tikzpicture}
\hspace{0.5cm}
\begin{tikzpicture}[>=latex,scale=1.65]
\draw[-](-0.35,0)--(1.75,0);
\draw[-](0,-0.35)--(0,2.25);
\draw[dotted](1,0)--(1,1);
\draw[samples=100](0,1)--(1,1);
\fill[](0,1)circle(1pt);
\fill[](1,1)circle(1pt);
\draw(0,1)node[left]{$1$};
\draw(1,0)node[below]{$1$};
\draw(0.75,-0.25)node[below]{$\widehat{v}$};
\end{tikzpicture}
\end{figure}

\item Lignola and Morgan~\cite[Propositions~4.3.1--4.3.2]{LM92} proved that hypo-convergence $v^k\stackrel{h}{\to} v$ holds under each of the following assumptions that differ from ours$:$
\begin{description}
\item[$(i)$] $h\mbox{-}\!\limsup\nolimits_k u^k\le u$, $\Phi^k\stackrel{c}{\to}\Phi$ with $\{\Phi^k\}$ env and elb, and for any $(x,y)$ there exists $\tilde x^k\to x$ such that $u(x,y)\le\liminf\nolimits_k u^k(\tilde x^k,y^k)$ for every $y^k\to y$.

\item[$(ii)$] $u^k\stackrel{h}{\to} u$, $g\mbox{-}\!\limsup\nolimits_k\Phi^k\subset \Phi$ with $\{\Phi^k\}$ env, and $\{\Phi^k\}$ open graph converges to $\Phi$.
\end{description}
Concerning the hypothesis of Theorem~\ref{MT1}$(b)$, we see that in $(i)$, assumption $u\le e\mbox{-}\!\liminf\nolimits_k u^k$ is weakened and the
    convergence notion for feasible mappings is strengthened. Whereas in $(ii)$, the convergence notion for objective functions is weakened, condition elb is dropped, and
    assumption $\Phi\subset g\mbox{-}\!\liminf\nolimits_k\Phi^k$ is strengthened.
\end{enumerate}
\end{remark}

\section{Applications}\label{sec:4}
We apply the above results to determine the stability properties of two variational problems. Specifically, we examine generalized Nash equilibrium problems and finite-horizon dynamic programming models.

\subsection{Stability of generalized Nash equilibrium problems}
We study the stability of the generalized Nash equilibrium problem (GNEP) that is an extension of the Nash equilibrium problem (NEP).
Contrary to optimization problems, the literature on approximating GNEPs is
not as extensive. Morgan and Raucci~\cite{MR99},
G\"{u}rkan and Pang~\cite{GP09}, and Diem and Khanh~\cite{DK16} primarily focused on approximating NEPs. There is relatively limited exploration in the realm
of GNEPs. Building upon the work of Morgan and Scalzo~\cite{MS08}, we delve into
investigating the convergence of GNEPs. Royset and Wets~\cite{RW19} and Diem and Khanh~\cite{DK24} studied the stability of GNEPs
by using the Nikaido--Isoda bifunction that characterizes the solution of a GNEP. Our approach is straightforward without using such a bifunction.

Let $N$ be a nonempty and finite set of players. Let us assume that each player, labeled by $\nu \in N$, chooses a strategy $x_\nu $ in a strategy set $K_{\nu}\subset\R^{n_\nu}$. We define the Cartesian products $\R^n:=\prod_{\nu\in N}\R^{n_\nu}$ where $n=\sum_{\nu\in N}n_\nu$, $K:=\prod_{\nu\in N} K_{\nu}$, and
$K_{-\nu}:=\prod_{\mu\in N\setminus\{\nu\}} K_{\mu}$ for $\nu\in N$. We write $x = (x_\nu, x_{-\nu}) \in K$ to emphasize the strategy $x_\nu\in K_\nu$ of player $\nu$ and the strategy $x_{-\nu}\in K_{-\nu}$ of the other players.

Given the strategy $x_{-\nu}$ of the players except of player $\nu$, the player $\nu$ chooses a strategy $x_\nu$ solving the following problem:
\begin{equation}\label{NEP}
\max_{x_\nu\in K_\nu} \theta_\nu(x_\nu,x_{-\nu}),
\end{equation}
where $\theta_\nu\colon \R^n\to\R$ is a real-valued function and  $\theta_\nu(x_\nu,x_{-\nu})$ denotes the payoff of player $\nu$ when the rival players have chosen
the strategy $x_{-\nu}$.  A vector $\hat{x}\in K$  is a Nash equilibrium, if $\hat{x}_\nu$ solves~\eqref{NEP} when its rival players take the strategy
$\hat{x}_{-\nu}$ for every $\nu\in N$  (see~\cite{Na51}). We denote by $\nep(\{\theta_\nu,K_\nu\}_{\nu\in N})$ the set of Nash equilibria.

The necessity of a generalization of the NEP arose when involving player interactions at the feasible sets level.  Arrow and Debreu~\cite{AD54} termed it as
abstract economy, but nowadays it is termed the GNEP. Recently, it gained more and more attention because it models electricity
markets, environmental games, bilateral exchanges of bads, among others (see for instance~\cite{AuEtAl16,CoEtAl15,KT16,VZ19}).

Formally, in a GNEP (see~\cite{FaEtAl10}), each player's strategy $x_\nu$ must belong to a set $X_\nu(x_{-\nu})\subset K_\nu$ depending on the rival players'
strategies. The aim of player $\nu$, given the others players' strategies $x_{-\nu}$, is to choose a strategy $x_\nu$ that solves the next maximization problem
\begin{equation}\label{GNEP}
\max_{x_\nu\in  X_\nu(x_{-\nu})} \theta_\nu(x_\nu,x_{-\nu}),
\end{equation}
where $X_\nu\colon\R^{n-n_{\nu}}\tos \R^{n_\nu}$ is a multifunction such that $X_\nu(\R^{n-n_\nu})\subset K_\nu$. A vector $\hat{x}\in K$ is a generalized Nash
equilibrium, if $\hat{x}_\nu$ solves~\eqref{GNEP} when its rival players take the strategy $\hat{x}_{-\nu}$ for every $\nu\in N$; i.e.,
\[
\hat{x}_\nu\in\argmax_{x_\nu\in  X_\nu(\hat{x}_{-\nu})} \theta_\nu(x_\nu,\hat{x}_{-\nu}),\;\,\forall \nu\in N.
\]
We denote by $\gnep(\{\theta_\nu,X_\nu\}_{\nu\in N})$ the set of generalized Nash equilibria.

For each player $\nu\in N$, we define the best response multifunction $S_\nu\colon\R^{n-n_{\nu}}\tos K_\nu$ and the function $v_\nu\colon\R^{n-n_\nu}\to\overline{\R}$ by
\[
S_\nu(x_{-\nu}):=\argmax_{x_\nu\in X_\nu(x_{-\nu})}\theta_\nu(\cdot,x_{-\nu})\;\;\mbox{ and }\;\;v_\nu(x_{-\nu}):=\sup_{x_\nu\in X_\nu(x_{-\nu})}~\theta_\nu(x_\nu,x_{-\nu}).
\]
It is not difficult to verify that
\begin{equation}\label{eq:formula-gnep}
\gnep(\{\theta_\nu,X_\nu\}_{\nu\in N})=\bigcap\nolimits_{\nu\in N}\gra S_\nu.
\end{equation}
We establish convergence properties of the approximations $\{\theta_{\nu}^k\}_{\nu\in N}$ and $\{X_{\nu}^k\}_{\nu\in N}$ of the objective function and of the feasible mapping.

\begin{theorem}\label{GNEP1}
Let $\theta_{\nu}^k\stackrel{c}{\to} \theta_\nu$ and $X_{\nu}^k\stackrel{c}{\to} X_\nu$, for every $\nu\in N$. Then
\begin{description}
\item [$(a)$] $\limsup_k \gnep(\{\theta^k_\nu,X^k_\nu\}_{\nu\in N})\subset \gnep(\{\theta_\nu,X_\nu\}_{\nu\in N})$.
\item [$(b)$] If  $\{X_{\nu}^k\}$ is elb and env for every $\nu\in N$, then $v_{\nu}^k\stackrel{c}{\to} v_\nu$ for every $\nu\in N$.
\end{description}
\end{theorem}

\begin{proof} $(a)$ By Theorem~\ref{MT0}$(a)$ and Remark~\ref{rem:continuous2}, we have $\limsup_k (\gra S_\nu^k)\subset  \gra S_\nu$ for all $\nu\in N$. From this and~\eqref{eq:formula-gnep}, we deduce that
\begin{align*}
\limsup\nolimits_k \gnep(\{\theta^k_\nu,X^k_\nu\}_{\nu\in N})
&=\limsup\nolimits_k\bigcap\nolimits_{\nu\in N}\gra\, S^k_\nu\\
&\subset \bigcap\nolimits_{\nu\in N}\limsup\nolimits_k(\gra\, S_\nu^k)\\
&\subset \bigcap\nolimits_{\nu\in N} \gra\, S_\nu\\
&=  \gnep(\{\theta_\nu,X_\nu\}_{\nu\in N}).
\end{align*}

$(b)$ It follows from Theorem~\ref{MT0}$(b)$.
\end{proof}

As a consequence of Theorem~\ref{GNEP1}, we recover \cite[Proposition~9]{DK16} and \cite[Proposition~9]{DK24}.

\begin{corollary}\label{cor:nep}
If $\theta_{\nu}^k\stackrel{c}{\to} \theta_\nu$ and $K_\nu^k\to K_\nu$ for every $\nu\in N$, then
\[
\limsup\nolimits_k \nep(\{\theta^k_\nu,K^k_\nu\}_{\nu\in N})\subset \nep(\{\theta_\nu,K_\nu\}_{\nu\in N}).
\]
\end{corollary}

\begin{remark}\label{rem:GNEP}
Concerning Theorem~\ref{GNEP1} and Corollary~\ref{cor:nep}, a few remarks are needed.
\begin{enumerate}
\item Royset and Wets~\cite[Proposition~4.12]{RW19}
inferred part $(a)$ of Theorem~\ref{GNEP1} through the use of approximating Nikaido--Isoda bifunctions.
\item Corollary~\ref{cor:nep} is closely related to \cite[Theorem~1]{GP09}.
Instead of
using multi-epi\-convergence as in~\cite{GP09},
we use continuous convergence. Moreover, in contrast to \cite[Theorem~1]{GP09}, our sequence of constraint sets is not necessarily constant.
\item Corollary~\ref{cor:nep} is also related to \cite[Theorem~1]{MS08}. In place of the pseudocontinuity framework employed in~\cite{MS08}, our analysis is conducted within the setting of continuous convergence.
\end{enumerate}
\end{remark}

A natural question is whether, under the hypothesis of Theorem~\ref{GNEP1}, one has
$$\gnep(\{\theta_\nu,X_\nu\}_{\nu\in N})\subset \liminf\nolimits_k\gnep(\{\theta_\nu^k,X^k_\nu\}_{\nu\in N})?$$
The answer to this question is negative, as shown in~\cite[Example~1.1]{MR99} for a {\rm NEP}.

\subsection{Stability of finite-horizon dynamic programming models}

We study the stability of the finite-horizon discrete-time discounted dynamic programming model under certainty.
Dynamic programming models have been widely used by a number of authors in various well-known papers on economic theory as Arrow et al.~\cite{ArEtAl51},  Brock and
Mirman~\cite{BM72}, Kydland and Prescott~\cite{KP82}, Lucas~\cite{Lu78}, and Lucas and Prescott~\cite{LP71}, among others. Discrete dynamic programming  models have often been useful to address some discrete optimal control problems as done by Guigue et al.~\cite{GuEtAl09}, Ha et al.~\cite{HaEtAl21},  Murray and Yakowitz~\cite{MY81}, and
Ramadge and Wonham~\cite{RW23}, among others.

We deal with a finite-horizon version of the model formulated in Stokey and Lucas~\cite{SL89}. This model is formulated as follows: for a given  $y\in\R^m$, we want to find $y^*_1,\dots,y^*_T\in \R^{m}$ with   $y^*_{j+1}\in\Gamma(y_j^*)$ for all $j\in\{0,\dots,T-1\}$, where $y^*_0=y$, such that
\begin{equation}\label{eq:model-dpm}
\sum_{j=0}^{T-1}\beta^jv(y^*_j,y^*_{j+1})=\sup\left\{
\sum_{j=0}^{T-1}\beta^jv(y_j,y_{j+1})\colon
\begin{matrix}
y_1,\dots,y_{T}\in \R^m,\\
y_0=y,\; y_{j+1}\in\Gamma(y_j),\\
\forall j\in\{0,\dots,T-1\}
\end{matrix}
\right\}.
\end{equation}
Here $\beta\in\,]0,1]$ is the discount rate, $v:\R^m\times\R^m\to \R$ is the return function, and $\Gamma:\R^m\tos \R^m$ is a nonempty-valued multifunction describing the constraints.

Goberna and Todorov~\cite{GT09} studied the stability of a dynamic programming
model where the time-dependent discount rate is perturbed, whereas the return function and the feasibility constraints remain
fixed and are assumed to be linear. On the contrary, our stability analysis assumes that the discount rate is constant, whereas the return function and the multifunction describing the feasibility constraints are perturbed. Granzotto et al.~\cite{G-etal21} introduced a model related to model~\eqref{eq:model-dpm} with a bounded time-dependent return function. They studied the stability of their model; however, because it is unconstrained, no perturbed feasible constraints are required, in contrast to our analysis.

To translate model~\eqref{eq:model-dpm} to our framework, we introduce the following multifunction $\Phi_\Gamma\colon\R^m\tos \R^{m\times T}$ defined as
\[
\Phi_\Gamma(y) :=\left\lbrace {\bf x}\in \R^{m\times T} \colon
\begin{matrix}
{\bf x}=(y_1,y_2,\dots,y_T)\mbox{ and }\\
y_{j+1}\in\Gamma(y_j),\forall j\in\{0,\dots,T-1\}\\
\mbox{with } y_0=y
\end{matrix}
\right\rbrace
\]
and the function $f:\R^{m\times T}\times \R^m\to\R$ as
\[
f({\bf x},y):=v(y,y_1)+\sum_{j=1}^T\beta^jv(y_j,y_{j+1}), \mbox{ where }{\bf x}=(y_1,y_2,\dots,y_T).
\]
Thus,
the value function $\mu\colon\R^m\to\R\cup\{+\infty\}$ is given by
\[
\mu(y)=\sup_{{\bf x}\in \Phi_\Gamma(y)}f({\bf x},y),
\]
and the solution mapping $\Lambda\colon \R^m\tos\R^{m\times T}$ is
$$
 \Lambda(y) = \left\{{\bf x}\in \Phi_\Gamma(y)\colon \mu(y)=\sum_{j=0}^{T-1}\beta^jv(y_j,y_{j+1}),\mbox{ with }y_0=y\right\}.
$$

\begin{remark}
According to the Berge theorem, if $v$ is continuous and $\Phi_{\Gamma}$ is nonempty-valued, compact-valued, $K$-continuous, then $\mu$ is continuous and $\Lambda$ is nonempty-valued, compact-valued, and usc. Thus, continuity of $\Phi_\Gamma$ can be guaranteed by continuity of $\Gamma$. In addition, if $\Gamma$ is single-valued, then $\Phi_\Gamma$ and $\Lambda$ are single-valued.
\end{remark}

To perturb model~\eqref{eq:model-dpm}, we consider a nonempty-valued mapping $\Gamma_k\colon \R^m\tos \R^m$  describing the approximate feasible mappings, $v^k\colon\R^m\times\R^m\to \R$ is the approximate return function, and $\mu^k\colon \R^m\to
\R$ is the approximate value function,
for every $k\in\mathbb{N}$.
For each $k\in\nn$, we define  $\Phi_{\Gamma_k}:\R^m\tos \R^{m\times T}$ as
\[
\Phi_{\Gamma_k}(y) :=\left\lbrace {\bf x}\in \R^{m\times T} \colon
\begin{matrix}
{\bf x}=(y_1,y_2,\dots,y_T)\mbox{ and }\\
y_{j+1}\in\Gamma_k(y_j),\forall j\in\{0,\dots,T-1\}\\
\mbox{with } y_0=y
\end{matrix}
\right\rbrace
\]
and the approximate solution mapping $\Lambda^k\colon\R^m\tos\R^{m\times T}$ by
\[
\Lambda^k(y) = \left\{{\bf x}\in \Phi_{\Gamma_k}(y)\colon \mu^k(y)=\sum_{j=0}^{T-1}\beta^jv^k(y_j,y_{j+1}),\mbox{ with }y_0=y\right\}.
\]

We consider the sequence of functions $f^k$, associated to function $f$, defined as
\[
f^k({\bf x},y):=v^k(y,y_1)+\sum_{j=1}^T\beta^jv^k(y_j,y_{j+1}), \mbox{ where }{\bf x}=(y_1,y_2,\dots,y_T).
\]

We derive properties of the sequence of feasible mappings $\{\Phi_{\Gamma_ k}\}$. These properties complement the convergence results for feasible mappings in Propositions~\ref{prop:ex-feas-mappings}--\ref{l2}. To do this, we define the set $\Theta:=\R^m\times\{(w,w)\colon w\in \R^{m\times(T-1)}\}\times\R^m$ and the following mappings for $j\in\{1,\ldots ,T\}$:
$$\Gamma^{T}:=\underbrace{\Gamma\times \cdots\times \Gamma}_{\mbox{$T$-times}}\;\; \mbox{ and }\;\;\Gamma^{(j)}:=\underbrace{\Gamma\circ \cdots\circ \Gamma}_{\mbox{$j$-times}}.$$
Note that $\Phi_\Gamma(y)\subset \Gamma^{(1)}(y)\times\cdots\times\Gamma^{(T)}(y)$.

\begin{proposition}\label{prop:propertie-phi-gamma} Let $\Gamma$ be nonempty-valued and $\{\Gamma_k\}$ be env. Then
\begin{description}
\item[$(a)$] If $\Gamma_ k \stackrel{c}{\to} \Gamma$, then $\Phi_{\Gamma_ k} \stackrel{c}{\to}\Phi_{\Gamma}$.
\item[$(b)$] If $g\mbox{-}\!\limsup\nolimits_k\Gamma_k\subset\Gamma$, then $g\mbox{-}\!\limsup\nolimits_k\Phi_{\Gamma_ k} \subset\Phi_{\Gamma}$.
\item[$(c)$] If $\operatorname{gph}(\Gamma^{T})\cap\Theta\subset\liminf\nolimits_k(\operatorname{gph}(\Gamma_k^{T})\cap\Theta)$, then $\Phi_{\Gamma}\subset g\mbox{-}\!\liminf\nolimits_k\Phi_{\Gamma_ k}$.
\item[$(d)$] If $\{\Gamma_ k\}$ is eub, then $\{\Phi_{\Gamma_k}\}$ is eub.
\item[$(e)$] If $\{\Gamma_ k^{(j)}\}$ is elb for every $j\in\{1,\ldots ,T\}$, then $\{\Phi_{\Gamma_k}\}$ is elb.
\item[$(f)$] $\{\Phi_{\Gamma_k}\}$ is env.
\end{description}
\end{proposition}

\begin{proof} $(a)$ Let $y^k\to y$ be fixed. First, we check that $\Phi_{\Gamma}(y)\subset\liminf_k \Phi_{\Gamma_ k}(y^k)$. Let ${\bf x}\in\Phi_{\Gamma}(y)$. As $\Gamma(y)\subset\liminf_k \Gamma_k(y^k)$, there exists $y^k_1\in\Gamma_k(y^k)\to y_1\in \Gamma(y)$. As $\Gamma(y_1)\subset\liminf_k\Gamma_k(y^k_1)$, similarly, there exists $y^k_2\in\Gamma_k(y^k_1)\to y_2\in\Gamma(y_1)$. Proceeding inductively, we prove that there exists $y^k_T\in\Gamma_k(y^k_{T-1})\to y_T\in \Gamma(y_{T-1})$. Hence ${\bf x}^k:=(y_1^k,\ldots ,y_T^k)\in \Phi_{\Gamma_k}(y^k)\to {\bf x}$; i.e., ${\bf x}\in\liminf_k \Phi_{\Gamma_k}(y^k)$ and the inclusion follows. Similarly, we can check that $\limsup_k \Phi_{\Gamma_k}(y^k)\subset \Phi_\Gamma(y)$. Therefore, $\Phi_{\Gamma_k}(y^k)\to\Phi_\Gamma(y)$ and since $y^k\to y$ was arbitrary, we conclude that $\Phi_{\Gamma_ k} \stackrel{c}{\to}\Phi_{\Gamma}$.

$(b)$ We check that $\limsup_k\operatorname{gph}\Phi_{\Gamma_k}\subset \operatorname{gph}\Phi_\Gamma$. If $(y,{\bf x})$ is in the left-side set, then there exists $(y^{k_\ell},{\bf x}^{k_\ell})\in\operatorname{gph}\Phi_{\Gamma_{k_\ell}}\to (y,{\bf x})$. Since ${\bf x}^{k_\ell}=(y^{k_\ell}_1,\ldots ,y^{k_\ell}_T)$, $y^{k_\ell}_{j+1}\in\Gamma_{k_\ell}(y^{k_\ell}_{j})$ with $y^{k_\ell}_{0}=y^{k_\ell}$ and $(y^{k_\ell}_{j},y^{k_\ell}_{j+1})\in \operatorname{gph}\Gamma_{k_\ell}\to (y_j,y_{j+1})$ for $j\in\{0,1,\ldots ,T-1\}$, we have $(y_j,y_{j+1})\in\limsup_k\operatorname{gph}\Gamma_{k}$ that by hypothesis implies $(y_j,y_{j+1})\in \operatorname{gph}\Gamma$ for such $j$. Hence $(y,{\bf x})\in \operatorname{gph}\Phi_\Gamma$.

$(c)$ If $(y,{\bf x})\in\operatorname{gph}\Phi_\Gamma$, then ${\bf x}\in\Phi_\Gamma(y)$ where ${\bf x}=(y_1,\ldots ,y_T)$ and $y_{j+1}\in\Gamma(y_{j})$ for $j\in\{0,1,\ldots ,T-1\}$ with $y_{0}=y$. Hence $(y_1,\ldots ,y_T)\in \Gamma^{T}(y,y_1,\ldots ,y_{T-1})$, so by setting $z:=(y_1,\ldots ,y_{T-1})$, we have $(y,z,z,y_T)\in \operatorname{gph}(\Gamma^{T})\cap\Theta$. By hypothesis there exists  $(y^k,z^k,z^k,y_T^k)\in \operatorname{gph}(\Gamma_k^{T})\cap\Theta\to (y,z,z,y_T)$ where $z^k=(y_1^k,\ldots ,y_{T-1}^k)$. Let ${\bf x}^k:=(y_1^k,\ldots ,y_T^k)$. As $y_{j+1}^k\in\Gamma_k(y_{j}^k)$ for $j\in\{0,1,\ldots ,T-1\}$ with $y_{0}^k=y^k$, we have $(y^k,{\bf x}^k)\in\operatorname{gph}\Phi_{\Gamma_k}\to (y,{\bf x})$. Hence $(y,{\bf x})\in\liminf_k \operatorname{gph}\Phi_{\Gamma_k}$.

$(d)$ As $\{\Gamma_k\}$ is eub, there exist $N\in\mathbb{N}$ and $r>0$ such that $\Gamma_k(\R^m)\subset r\mathbb{B}$ for all $k\ge N$. Let us take arbitrary $y\in\R^m$ and $k\ge N$. If ${\bf x}\in \Phi_{\Gamma_k}(y)$, then ${\bf x}=(y_1^k,\ldots ,y_T^k)$ with
$y_1^k\in\Gamma_k(y)$, and $y^k_{j+1}\in\Gamma_k(y_j^k)$ for $j\in\{1,2,\ldots, T-1\}$. Hence ${\bf x}\in r\mathbb{B}\times\cdots\times r\mathbb{B}$ for all $k\ge N$ and since $y$ was arbitrary, we have $\Phi_{\Gamma_k}(\R^m)\subset r\mathbb{B}\times\cdots\times r\mathbb{B}$ for all $k\ge N$.

$(e)$ Let $y\in\R^m$ be arbitrary and $j\in\{1,\ldots ,T\}$. As $\{\Gamma_k^{(j)}\}$ is elb at $y$, there exists $N_j\in\mathbb{N}$, a neighborhood $U_j$ of $y$, and $r_j>0$ such that $\Gamma_k^{(j)}(U_j)\subset r_j\mathbb{B}$ for all $k\ge N_j$. Thus, for $N=\max_{1\le j\le T}N_j$ and $U=\cap_{j=1}^{T}U_j$, we have
$$\Phi_{\Gamma_k} (U)\subset\Gamma_k^{(1)}(U)\times\cdots\times\Gamma_k^{(T)}(U)\subset r_1\mathbb{B}\times\cdots\times r_{T}\mathbb{B},\;\forall k\ge N.$$
Hence $\Phi_{\Gamma_k}$ is elb at $y$ and the result follows since $y$ was arbitrary.

$(f)$ It follows straightforwardly.
\end{proof}

We establish convergence properties of the approximations of the objective function and of the feasible mapping.

\begin{theorem}\label{d1}
Let $\Gamma$ be nonempty-valued, $\{\Gamma_k\}$ be env, $v^ k \stackrel{c}{\to} v$ and $\Gamma_ k \stackrel{c}{\to} \Gamma$. Then
\begin{description}
\item[$(a)$] $g\mbox{-}\!\limsup\nolimits_ k  \Lambda^ k\subset \Lambda$.
\item[$(b)$] Let $\{\Gamma_ k^{(j)}\}$ be elb for every $j\in\{1,\ldots ,T\}$. Then
\begin{description}
\item[$(i)$] $\mu^ k \stackrel{c}{\to} \mu$, $\{\Lambda^k\}$ is elb, and $\Lambda$ is nonempty-valued, locally bounded, osc.
\item[$(ii)$] If $\Gamma_k$ is osc and $v^k$ is usc for all $k$, then
$\{\Lambda^k\}$ is selb and $g\mbox{-}\!\limsup\nolimits_ k  \Lambda^ k$ is nonempty-bounded-valued.
\end{description}
\end{description}
\end{theorem}

\begin{proof} Clearly, $v^ k \stackrel{c}{\to} v$ implies $f^ k \stackrel{c}{\to} f$. On the other hand, $\Gamma_ k \stackrel{c}{\to} \Gamma$ implies $\Phi_{\Gamma_ k} \stackrel{c}{\to}\Phi_{\Gamma}$ by Proposition~\ref{prop:propertie-phi-gamma}$(a)$.

$(a)$  It follows from Theorem~\ref{MT0}$(a)$.

$(b)$ As $\{\Phi_{\Gamma_k}\}$ is elb by Proposition~\ref{prop:propertie-phi-gamma}$(e)$, the first part of item~$(i)$ follows from Theorem~\ref{MT0}$(b)$. As $\Phi_\Gamma$ is locally bounded by Proposition~\ref{prop:elb-cont}$(a)$, $\mu$ and $\Phi_\Gamma$ are continuous by Remarks~\ref{rem:h-conv}$(1)$ and~\ref{rem:continuous2}$(2)$, the second part follows from Corollary~\ref{cor:berge-variant}$(c)$. Part~$(ii)$ follows from Proposition~\ref{prop:propertie-phi-gamma}$(f)$, Theorem~\ref{MT0}$(c)$ and since under the hypotheses each $f^k$ is usc and $\Phi_{\Gamma_k}$ is closed-valued.
\end{proof}

\begin{theorem}\label{dd}
Let $\Gamma$ be nonempty-valued, $\{\Gamma_k\}$ be env, $v^ k \stackrel{c}{\to} v$, $g\mbox{-}\!\limsup\nolimits_k\Gamma_k\subset\Gamma$  and $\operatorname{gph}(\Gamma^{T})\cap\Theta\subset\liminf\nolimits_k(\operatorname{gph}(\Gamma_k^{T})\cap\Theta)$. Then
\begin{description}
    \item[$(a)$] If $\{\Gamma_k\}$ is eub, then $\{\Lambda^k\}$ is eub and~$\Lambda$ is bounded, compact-valued.
    \item[$(b)$] Let $\{\Gamma_ k^{(j)}\}$ be elb for every $j\in\{1,\ldots ,T\}$. Then
    \begin{description}
    \item[$(i)$] $\mu^ k \stackrel{h}{\to} \mu$ and $\{\Lambda^k\}$ is elb.
    \item[$(ii)$] $\forall y\in\R^m$, $\exists y^k\to y$; $\mu^ k(y^k) \to \mu(y)$ and $\limsup_k \Lambda^k(y^k)\subset \Lambda(y)$.
     \item[$(iii)$] If  $\Gamma_k$ osc and $v^k$ usc for all $k$, then for all  $y^k\to y$, $\limsup\nolimits_k \Lambda^k(y^k)$ is nonempty compact and $(g\mbox{-}\!\limsup\nolimits_k \Lambda^k)(y)\cap \Lambda(y)\ne\emptyset$.
    \end{description}
\end{description}
 \end{theorem}

\begin{proof} Clearly, $v^ k \stackrel{c}{\to} v$ implies $f^ k \stackrel{c}{\to} f$. On the other hand, we have $\Phi_{\Gamma_ k} \stackrel{g}{\to}\Phi_{\Gamma}$ by Proposition~\ref{prop:propertie-phi-gamma}$(b)$--$(c)$. Part~$(a)$ follows from Proposition~\ref{prop:propertie-phi-gamma}$(d)$ and Theorem~\ref{MT1}$(a)$. Part~$(b)$ follows from Proposition~\ref{prop:propertie-phi-gamma}$(e)$--$(f)$, Theorem~\ref{MT1}$(b)$, Corollary~\ref{cor:stability-c-g}, and since under the hypotheses of $(iii)$, each $f^k$ is usc and $\Phi_{\Gamma_k}$ is closed-valued.
\end{proof}

\section{Conclusions}\label{sec:5}

Motivated by Berge's maximum theorem, we have investigated the stability of the value function and the solution mapping associated with a parametric optimization problem under data perturbations. To approximate the objective function and the feasible mapping, we employed lower and upper continuous convergence, as well as epi- and hypo-convergence notions for functions, together with lower and upper continuous and graphical convergence notions for multifunctions. The choice of these convergence concepts is particularly advantageous, as they have been extensively studied in the literature; see, for instance,~\cite{At84,HP97,RW09,RW21} where the stability of fundamental mapping operations has been thoroughly analyzed.

Within this framework, we have established convergence properties for several classes of multifunctions, including interval, unexplicit, and inequality mappings, as well as for their compositions and Cartesian products. These results provide a flexible and unified variational setting for studying the stability of parametric optimization problems beyond classical compactness and strong continuity assumptions.

Finally, we applied our theoretical findings to generalized Nash equilibrium problems and finite-horizon dynamic programming models. In both applications, the proposed approach yields direct stability results for equilibria and value functions, thereby highlighting the practical relevance of the variational convergence framework for analyzing approximations of equilibria and optimal policies in the presence of data uncertainty.

\section*{Statements and declarations}

\noindent\textbf{Funding.} This paper has received funding from ANID--Chile through projects Fondecyt 1220687 (L\'opez) and Fondecyt 1200525 (Fierro).

\medskip
\noindent\textbf{Conflict of interest.} The authors declare that they have no conflict of interest.


\begin{thebibliography}{99}

\bibitem{AD54}
Arrow, K.J., Debreu, G.: Existence of an equilibrium for a competitive economy.
Econometrica \textbf{22}(3), 265--290 (1954)

\bibitem{ArEtAl51}
Arrow, K.J., Harris, T., Marschak, J.: Optimal inventory policy.
Econometrica \textbf{19}(3), 250--272 (1951)

\bibitem{At84}
Attouch, H.: Variational Convergence for Functions and Operators.
Pitman, Boston (1984)

\bibitem{AuEtAl16}
Aussel, D., Gupta, R., Mehra, A.: Evolutionary variational inequality
formulation of the generalized Nash equilibrium problem.
J. Optim. Theory Appl. \textbf{169}(1), 74--90 (2016)

\bibitem{Be93}
Beer, G.: Topologies on Closed and Closed Convex Sets. Kluwer, Dordrecht (1993)

\bibitem{Be63}
Berge, C.: Topological Spaces.
Oliver and Boyd Ltd., London (1963)

\bibitem{BM72}
Brock, W.A., Mirman, L.J.: Optimal economic growth and uncertainty: the
discounted case.
J. Econom. Theory \textbf{4}(3), 479--513 (1972)

\bibitem{CaEtAl86}
Cavazzuti, E., Pacchiarotti, N.: Convergence of Nash equilibria.
Boll. Unione Mat. Ital. B \textbf{6}, 266--274 (1986)

\bibitem{CoEtAl15}
Contreras, J., Krawczyk, J.B., Zuccollo, J.: Economics of collective
monitoring: a study of environmentally constrained electricity generators.
Comput. Manag. Sci. \textbf{13}(3), 349--369 (2016)

\bibitem{CPS92}
Cottle, R.W., Pang, J.S., Stone, R.E.: The Linear Complementarity Problem.
SIAM, Philadelphia (2009)

\bibitem{DK16}
Diem, H.T., Khanh, P.Q.: Approximations of optimization-related problems
in terms of variational convergence.
Vietnam J. Math. \textbf{44}(2), 399--417 (2016)

\bibitem{DK24}
Diem, H.T.H., Khanh, P.Q.: Approximations of quasi-equilibria and Nash
quasi-equilibria in terms of variational convergence.
Set-Valued Var. Anal. \textbf{32}(1), 5 (2024)

\bibitem{FaEtAl10}
Facchinei, F., Kanzow, C.: Generalized Nash equilibrium problems.
Ann. Oper. Res. \textbf{175}(1), 177--211 (2010)

\bibitem{GT09}
Goberna, M.A., Todorov, M.I.: Primal-dual stability in continuous linear
optimization.
Math. Program. \textbf{116}(1), 129--146 (2009)

\bibitem{G-etal21}
Granzotto, M., Postoyan, R., Bu\c{s}oniu, L., Ne\v{s}i\'c, D., Daafouz, J.:
Finite-horizon discounted optimal control: stability and performance.
IEEE Trans. Automat. Control \textbf{66}(2), 550--565 (2021)

\bibitem{GuEtAl09}
Guigue, A., Ahmadi, M., Hayes, M.J.D., Langlois, R.G.: A discrete dynamic
programming approximation to the multiobjective deterministic finite horizon
optimal control problem.
SIAM J. Control Optim. \textbf{48}(4), 2581--2599 (2009)

\bibitem{GP09}
G{\"u}rkan, G., Pang, J.S.: Approximations of Nash equilibria.
Math. Program. \textbf{117}(1-2), 223--253 (2009)

\bibitem{HaEtAl21}
Ha, M., Wang, D., Liu, D.: Generalized value iteration for discounted optimal
control with stability analysis.
Systems Control Lett. \textbf{147}, 104847 (2021)

\bibitem{HP97}
Hu, S., Papageorgiou, N.S.: Handbook of Multivalued Analysis, Vol. I: Theory.
Springer, Dordrecht (1997)

\bibitem{KT16}
Krawczyk, J.B., Tidball, M.: Economic problems with constraints: how efficiency
relates to equilibrium.
Int. Game Theory Rev. \textbf{18}(04), 1650011 (2016)

\bibitem{KP82}
Kydland, F.E., Prescott, E.C.: Time to build and aggregate fluctuations.
Econometrica \textbf{50}(6), 1345--1370 (1982)

\bibitem{LM92}
Lignola, M.B., Morgan, J.: Convergences of marginal functions with dependent
constraints.
Optimization \textbf{23}(3), 189--213 (1992)

\bibitem{LM97}
Lignola, M.B., Morgan, J.: Stability of regularized bilevel programming
problems. J. Optim. Theory Appl. \textbf{93}(3), 575--596 (1997)

\bibitem{LS21}
L{\'o}pez, R., Sama, M.: Horizon maps and graphical convergence revisited.
SIAM J. Optim. \textbf{31}(2), 1330--1351 (2021)

\bibitem{Lu78}
Lucas, R.E.: Asset prices in an exchange economy.
Econometrica \textbf{46}(6), 1429--1445 (1978)

\bibitem{LP71}
Lucas, R.E., Prescott, E.C.: Investment under uncertainty.
Econometrica \textbf{39}(5), 659--681 (1971)

\bibitem{MR99}
Morgan, J., Raucci, R.: New convergence results for Nash equilibria.
J. Convex Anal. \textbf{6}(2), 377--385 (1999)

\bibitem{MS08}
Morgan, J., Scalzo, V.: Variational stability of social Nash equilibria.
Int. Game Theory Rev. \textbf{10}(01), 17--24 (2008)

\bibitem{MY81}
Murray, D.M., Yakowitz, S.J.: The application of optimal control methodology to
nonlinear programming problems.
Math. Program. \textbf{21}(1), 331--347 (1981)

\bibitem{Na51}
Nash, J.: Non-cooperative games.
Ann. Math. \textbf{54}(2), 286--295 (1951)

\bibitem{RW23}
Ramadge, P.J., Wonham, W.M.: Supervisory control of a class of discrete event
processes.
SIAM J. Control Optim. \textbf{25}(1), 206--230 (1987)

\bibitem{RW09}
Rockafellar, R.T., Wets, R.J.-B.: Variational Analysis. Springer, New York
(2009)

\bibitem{RW19}
Royset, J.O., Wets, R.J.-B.: Lopsided convergence: an extension and its
quantification.
Math. Program. \textbf{177}(1-2), 395--423 (2019)

\bibitem{RW21}
Royset, J.O., Wets, R.J.-B.: An Optimization Primer. Springer, New York (2022)

\bibitem{SL89}
Stokey, N.L., Lucas, R.E.: Recursive Methods in Economic Dynamics.
Harvard University Press, Cambridge, MA (1989)

\bibitem{TW20}
Tammer, C., Weidner, P.: Scalarization and Separation by Translation Invariant
Functions with Applications in Optimization, Nonlinear Functional Analysis, and
Mathematical Economics. Springer, Berlin (2020)

\bibitem{VZ19}
Vardar, B., Zaccour, G.: Strategic bilateral exchange of a bad.
Oper. Res. Lett. \textbf{47}(4), 235--240 (2019)

\bibitem{Zo84}
Zolezzi, T.: Stability analysis in optimization.
In: Optimization and Related Fields, pp. 397--419. Springer, Berlin (1984)

\end{thebibliography}
\end{document}